\documentclass[11pt]{amsart}

\usepackage{color,amssymb}
\usepackage{enumerate}
\usepackage{graphicx}

\input xy
\xyoption{all}

\newcommand{\F}{\mathcal F}
\newcommand{\K}{\mathcal K}
\newcommand{\Z}{\mathcal Z}

\newcommand{\G}{\mathcal G}

\newcommand{\U}{\mathcal U}
\newcommand{\V}{\mathcal V}

\newcommand{\IN}{\mathbb N}
\newcommand{\IZ}{\mathbb Z}

\newcommand{\IR}{\mathbb R}

\newcommand{\IC}{\mathbb C}
\newcommand{\C}{\mathcal C}

\newcommand{\id}{\mathrm{id}}
\newcommand{\ret}{\mathsf r}

\newcommand{\e}{\varepsilon}
\newcommand{\w}{\omega}
\newcommand{\Ra}{\Rightarrow}

\newcommand{\Lip}{\mathrm{Lip}}
\newcommand{\diam}{\mathrm{diam}}

\newcommand{\dist}{\mathrm{dist}}

\newtheorem{theorem}{Theorem}[section]

\newtheorem{proposition}[theorem]{Proposition}
\newtheorem{lemma}[theorem]{Lemma}

\newtheorem{problem}[theorem]{Problem}
\newtheorem{corollary}[theorem]{Corollary}

\theoremstyle{definition}

\newtheorem{remark}[theorem]{Remark}
\newtheorem{example}[theorem]{Example}

\title{Embedding fractals in Banach, Hilbert or Euclidean spaces}

\author{Taras Banakh, Magdalena Nowak, Filip Strobin}
\address{T.Banakh: Ivan Franko National University of Lviv (Ukraine) and  Jan Kochanowski University in Kielce (Poland)}
\email{t.o.banakh@gmail.com}

\address{M.Nowak: Institute of Mathematics, Jan Kochanowski University in Kielce, ul. \'Swietokrzyska 15, 25-406 Kielce, Poland}
\email{magdalena.nowak805@gmail.com}
\address{F.Strobin: Institute of Mathematics, Lodz University of Technology, W\'olcza\'nska 215, 90-924 \L\'od\'z, Poland}
\email{filip.strobin@p.lodz.pl}

\date{}

\begin{document}
\begin{abstract}
By a {\em metric fractal} we understand a compact metric space $K$ endowed with a finite family $\F$ of contracting self-maps of $K$ such that $K=\bigcup_{f\in\F}f(K)$. If $K$ is a subset of a metric space $X$ and each $f\in\F$ extends to a contracting self-map  of $X$, then we say that $(K,\F)$ is a fractal in $X$. We prove that each metric fractal $(K,\F)$ is 
\begin{itemize}
\item isometrically equivalent to a fractal in the Banach spaces $C[0,1]$ and $\ell_\infty$;
\item bi-Lipschitz equivalent to a fractal in the Banach space $c_0$;
\item isometrically equivalent to a fractal in the Hilbert space $\ell_2$ if $K$ is an ultrametric space.
\end{itemize}
We prove that for a metric fractal $(K,\F)$ with the doubling property there exists $k\in\IN$ such that the metric fractal $(K,\F^{\circ k})$ endowed with the fractal structure $\F^{\circ k}=\{f_1\circ\dots\circ f_k:f_1,\dots,f_k\in\F\}$ is equi-H\"older equivalent to a fractal in a Euclidean space $\IR^d$. This result is used to prove our main result saying that each finite-dimensional compact metrizable space $K$ containing an open uncountable zero-dimensional space $Z$ is homeomorphic to a fractal in a Euclidean space $\IR^d$. For $Z$, being a copy of the Cantor set, this embedding result was proved by Duvall and Husch in 1992.
\end{abstract}
\maketitle

\section{Preliminaries}\label{s1}

In this paper we start a systematic study of fractals that are equivalent (in isometric, bi-Lipschitz, equi-H\"older or topological sense) to fractals in some special metric spaces (like Euclidean, Hilbert or Banach spaces).
Fractals are defined as attractors of iterated function systems consisting of  contracting self-maps of complete metric spaces.

A self-map $f:X\to X$ on a metric space $(X,d_X)$ is called
{\em contracting} if $\Lip(f)<1$, where 
$$\Lip(f):=\sup_{x\ne y}\frac{d_X(f(x),f(y))}{d_X(x,y)}$$is  the {\em Lipschitz constant} of $f$. 

By the classical Banach Contraction Principle, each contracting self-map $f:X\to X$ of a complete metric space $X$ has a unique fixed point $x_\infty$, which can be found as the limit of the Cauchy sequence $(x_n)_{n\in\w}$ starting with an arbitrary point $x_0$ and having $x_{n+1}=f(x_n)$ for all $n\in\w$.

Any finite family $\F$ of contracting self-maps on a complete metric space $X$ induces a contracting self-map $$\breve\F:\K(X)\to\K(X),\;\;\breve\F:K\mapsto\bigcup_{f\in\F}f(K),$$ of the hyperspace $\K(X)$ of non-empty compact subsets of $X$, endowed with the Hausdorff metric (see, {\cite{BKNNS}, \cite{B}, \cite{Ed}, \cite{Hut}}). By the Banach Contraction Principle, the map $\breve\F$ has a unique fixed point $K=\bigcup_{f\in\F}f(K)$, which can be  be found as the limit of the Cauchy sequence $(K_n)_{n\in\w}$ starting with an arbitrary non-empty compact set $K_0\subset X$ and having $K_{n+1}=\breve\F(K_n)$ for all $n\in\w$.

This unique fixed point $K$ is called the {\em attractor} of the system $\F$ (briefly the {\em $\F$-fractal}) in the complete metric space $X$. 
The $\F$-fractal $K$  carries a special structure called the {\em fractal structure}.

By definition, a {\em fractal structure} on a compact metric space $K$ is a finite family $\F$ of contracting self-maps of $K$ such that $K=\bigcup_{f\in\F}f(K)$. The number $$\Lip(\F):=\max_{f\in\F}\Lip(f)$$ is called the {\em Lipschitz constant} of the fractal structure $\F$.  It is easy to see that for any fractal structure $\F$ on a compact metric space $K$ and any $n\in\IN$ the family $\F^{\circ n}=\{f_1\circ\dots\circ f_n:f_1,\dots,f_n\in\F\}$ is a fractal structure on $K$ with Lipschitz constant $\Lip(\F^{\circ n})\le\Lip(\F)^n$. 

A compact metric space $K$ endowed with a fractal structure $\F$ will be called a {\em metric fractal}. In other words, a {\em metric fractal} is a pair $(K,\F)$ consisting of a compact metric space $K$ and a fractal structure $\F$ on $K$.
The space $K$ will be called the {\em underlying metric space} of the metric fractal $(K,\F)$.

Therefore, for any finite family $\F$ of contracting self-maps of a complete metric space $X$ its $\F$-fractal $K$ is the underlying metric space of the metric fractal $(K,\F{\restriction}K)$, where $\F{\restriction}K=\{f{\restriction}K:f\in\F\}$ and $f{\restriction}K$ is the restriction of $f$ to $K$. The pair $(K,\F{\restriction}K)$ will be called a {\em fractal} in the metric space $X$.

The notion of a metric fractal has also a topological counterpart. A system $\F$ of continuous self-maps of a topological space $X$ is called {\em topologically contracting} if for any open cover $\U$ of $X$ there exists a number $n$ such that for any functions $f_1,\dots,f_n\in\F$ the set $f_1\circ\dots\circ f_n(X)$ is contained in some set $U\in\U$. It is known \cite{BKNNS} that a finite system $\F$ of continuous self-maps on a compact Hausdorff space $X$ is topologically contracting if and only if for any infinite sequence $(f_n)_{n\in\w}\in\F^\w$ the intersection $\bigcap_{n\in\w}f_0\circ\dots\circ f_n(X)$ is a singleton.

A {\em topological fractal} is a pair $(X,\F)$ consisting of a compact Hausdorff topological space $X$ and a finite topologically  contracting system $\F$ of continuous self-maps such that $X=\bigcup_{f\in\F}f(X)$. The compact topological space $X$ will be called the {\em underlying topological space} of a topological fractal $(X,\F)$.

It is easy to see that each metric fractal is a topological fractal. By \cite{BN1} and \cite{NS}, there exists a topological fractal whose underlying topological space is not homeomorphic to the underlying topological space of any metric fractal. On the other hand, for any topological fractal $(X,\F)$ its underlying topological space $X$ is metrizable by a metric $d$ such that $d(f(x),f(y))<d(x,y)$ for any $f\in\F$ and any distinct points $x,y\in X$ (see \cite[6.2]{BKNNS} or \cite{MM}).

The aim of this paper is to recognize metric and topological fractals which are equivalent to fractals in some special metric spaces, like Banach, Hilbert or Euclidean spaces.
By {a} {\em Euclidean space} we understand any finite-dimensional Hilbert space; a {\em Hilbert space} is a Banach space whose norm is generated by an inner product. It is well-known that each Euclidean space is isometric to some space $\IR^n$ endowed with the standard Euclidean distance.

Now we define the notion of equivalence of fractals. Two topological fractals $(X,\F_X)$ and $(Y,\F_Y)$ are called {\em topologically equivalent} if there exist a homeomorphism  $h:X\to Y$ such that $\F_Y=\{h\circ f\circ h^{-1}:f\in\F_X\}$. If $(X,\F_X)$ and $(Y,\F_Y)$ are metric fractals and  the homeomorphism $h$ is an isometry (resp. $\e$-isometry, bi-Lipschitz, equi-H\"older, bi-H\"older) homeomorphism, then the fractals are called {\em isometrically} (resp. {\em $\e$-isometrically, bi-Lipschitz, equi-H\"older, bi-H\"older}) {\em equivalent}.

A homeomorphism $f:X\to Y$ between metric spaces $(X,d_X)$ and $(Y,d_Y)$ is called
\begin{itemize}
\item an {\em isometry} if $d_Y(f(x),f(y))=d_X(x,y)$ for all $x,y\in X$;
\item an {\em $\e$-sometry} for some $\e>0$ if\newline $\frac1{1+\e}\cdot d_X(x,y)\le d_Y(f(x),f(y))\le (1+\e)\cdot d_X(x,y)$ for all $x,y\in X$;
\item  {\em bi-Lipschitz} if there exist two positive real constants $c,C$ such that\newline $c\cdot d_X(x,y)\le d_Y(f(x),f(y))\le C\cdot d_X(x,y)$ for all $x,y\in X$;
\item {\em equi-H\"older} if there exist three positive real constants $c,C,\alpha$ such that\newline $c\cdot d_X(x,y)^\alpha\le d_Y(f(x),f(y))\le C\cdot d_X(x,y)^\alpha$ for all $x,y\in X$;
\item {\em bi-H\"older} if there exist four positive real constants $c,C,\alpha,\beta$ such that\newline $c\cdot d_X(x,y)^\alpha\le d_Y(f(x),f(y))\le C\cdot d_X(x,y)^\beta$ for all $x,y\in X$.
\end{itemize}
For these four types of homeomorphisms we have the implications
$$\centerline{isometry $\Ra$ $\e$-isometry $\Ra$ bi-Lipschitz $\Ra$ equi-H\"older $\Ra$ bi-H\"older.}
$$

In the next section we present some results on the (isometric, $\e$-isometric, bi-Lipschitz, equi-H\"older) topological equivalence of (metric) topological fractals to fractals in Banach, Hilbert or Euclidean spaces.  The most difficult result of this paper is Theorem~\ref{t:main} saying that each finite-dimensional compact metrizable space $X$ containing an open uncountable zero-dimensional subspace $Z$ is homeomorphic to a fractal in a Euclidean space. The particular case of Theorem~\ref{t:main} with $Z$ being homeomorphic to a the Cantor set was proved by Duvall and Husch  \cite{DH} in 1992.
\section{Embedding of metric fractals into (almost) universal metric spaces} 

A metric space $X$ is called {\em universal} if it contains an isometric copy of each compact metric space.

A classical example of a universal metric space is the Banach space $C[0,1]$ of continuous real-valued functions on $[0,1]$  {(see \cite{Ban})}.
By a result of Dutrieux and Lancien \cite{DL}, a Banach space $X$ is universal if and only if it contains a linear isometric copy of $C[0,1]$.
In particular, the Banach space $\ell_\infty$ of bounded sequences is universal.

\begin{theorem}\label{t:embU} Any metric fractal $(X,\F)$ is isometrically equivalent to a fractal in any universal complete metric space $U$.
\end{theorem}

\begin{proof} By a result of Isbell \cite{Isbell64} (mentioned on page 32 of \cite{BL}), the compact metric space $X$ is a subset of a compact metric space $\epsilon X$, which is an {\em absolute $1$-Lipschitz extensor}. The latter means that each Lipschitz map $f:A\to \epsilon X$ defined on a subset $A$ of a metric space $M$ has a Lipschitz extension $\bar f:M\to\epsilon X$ with the same Lipschitz constant $\Lip(\bar f)=\Lip(f)$.

Given a universal metric space $U$, we can identify the compact metric space  $\epsilon X$ with a subspace of $U$. So, $X\subset\epsilon X\subset U$. Since $\epsilon X$ is an absolute $1$-Lipschitz extensor, every map $f\in\F$ has a Lipschitz extension $\bar f:U\to\epsilon X\subset U$ with Lipschitz constant $\Lip(\bar f)=\Lip(f)<1$.

Then the function system $\bar\F:=\{\bar f:f\in\F\}$ consists of contracting self-maps of $U$ and has the set $X=\bigcup_{f\in\F}f(X)=\bigcup_{f\in\F}\bar f(X)$ as its attractor. So, the metric fractal $(X,\F)$ is isometric to the fractal $(X,\bar \F{\restriction}X)$ in $U$.
\end{proof}

	\begin{remark} Theorem~\label{r:embU} has a topological version, proved in \cite{BS}: {\em each topological fractal $(X,\F)$ is topologically equivalent to a topological fractal in a topological space $U$, containing a topological copy of any compact metrizable space.}
\end{remark}

A metric space $X$ is defined to be {\em almost universal} if for every $\e>0$, every compact metric space $K$ is $\e$-isometric to a subspace of $X$.
A classical example of an almost universal metric space
 is the Banach space $c_0$. This important fact was  proved by Kalton and Lancien in \cite{KL}. By \cite{DL}, the Banach space $c_0$ is not universal. 
 
 
\begin{theorem}\label{t:aU} For any $\e>0$, any metric fractal $(X,\F)$ is 
$\e$-isometrically equivalent (and hence bi-Lipschitz equivalent) to a fractal in any almost universal complete metric space $U$.
\end{theorem}

\begin{proof} Let $(U,d_U)$ be an almost universal metric space. Since $\Lip(\F)<1$, we can choose a positive real number $\delta<\e$ so small that $(1+\delta)^2\cdot\Lip(\F)<1$.

By \cite{Isbell64}, the compact metric space $X$ is a subspace of a compact metric space $\epsilon X$, which is an absolute $1$-Lipschitz extensor.
By the almost universality of the metric space $U$, there exists a $\delta$-isometric embedding $\varphi:\epsilon X\to U$, which means that $$\frac1{1+\delta}\cdot d_{\epsilon X}(x,y)\le d_U(\varphi(x),\varphi(y))\le(1+\delta)\cdot d_{\epsilon X}(x,y)\mbox{ \ for all $x,y\in \epsilon X$}.$$
The first inequality implies that for each $f\in\F$ the function $f\circ \varphi^{-1}:\varphi(X)\to X\subset\epsilon X$ has Lipschitz constant $\Lip(f{\circ}\varphi^{-1})\le (1+\delta)\cdot\Lip(f)$ and hence has a Lipschitz extension $\tilde f:U\to \epsilon X$ with Lipschitz constant $\Lip(\tilde f)=\Lip(f{\circ}\varphi^{-1})\le (1+\delta)\cdot\Lip(f)$. Then the map $\bar f=\varphi\circ\tilde f:U\to U$ has Lipschitz constant $\Lip(\bar f)\le (1+\delta)\cdot \Lip(\tilde f)\le (1+\delta)^2\cdot\Lip(f)<1$.

Therefore, the function system $\bar\F:=\{\bar f:f\in\F\}$ consists of contracting maps of $U$ and has the set $\varphi(X)=\bigcup_{f\in\F}\varphi\circ f(X)=\bigcup_{f\in\F}\bar f(\varphi(X))$ as its attractor. Moreover, $\{\bar f{\restriction}\varphi(X):\bar f\in\bar\F\}=
\{\varphi{\circ}f{\circ}\varphi^{-1}:f\in\F\}$, which means that the fractal $(X,\F)$ is $\e$-isometrically (and hence bi-Lipschitz) equivalent to the fractal $(\varphi(X),\bar\F{\restriction}\varphi(X))$ in $U$.
\end{proof}

Applying Theorems~\ref{t:embU} and \ref{t:aU} to the (almost) universal Banach spaces $\ell_\infty$, $C[0,1]$ (and $c_0$), we obtain the following embeddability theorem.

\begin{corollary} Any metric fractal $(X,\F)$ is
\begin{itemize}
\item isometrically equivalent to a fractal in the Banach space $C[0,1]$;
\item isometrically equivalent to a fractal in the Banach space $\ell_\infty$;
\item bi-Lipschitz equivalent to a fractal in the Banach space $c_0$.
\end{itemize}
\end{corollary}

It is well-known that for each infinite compact Hausdorff space $K$, the Banach space $C(K)$ of continuous real-valued functions on $K$ contains an isometric copy of the Banach space $c_0$, which implies that the Banach space $C(K)$ is almost universal. If the compact space $K$ is not scattered, then it admits a continuous map onto $[0,1]$, which implies that $C(K)$ contains an isometric copy of the Banach space $C[0,1]$ and hence $C(K)$ is universal. Combining these observations with Theorem~\ref{t:embU}, we obtain the following embedding result.

\begin{theorem} Let $K$ be a (non-scattered) infinite compact Hausdorff space. Any metric fractal $(X,\F)$ is (isometrically) bi-Lipschitz equivalent to a fractal in the Banach space $C(K)$.
\end{theorem}

\section{Embedding metric fractals into Hilbert spaces} 

In this section we discuss the problem of embeddings of metric fractals into Hilbert spaces. It is well-known that each Hilbert space $H$ is linearly isometric to the Hilbert space $\ell_2(\kappa)$ for a suitable cardinal $\kappa$. The Hilbert space $\ell_2(\kappa)$ consists of functions $f:\kappa\to\IR$ with $\sum_{x\in\kappa}|f(x)|^2<\infty$. The norm of $\ell_2(\kappa)$ is generated by the inner product $\langle f,g\rangle:=\sum_{x\in\kappa}f(x)g(x)$. 

By the classical Kirszbraun Theorem 1.12 in \cite{BL}, any Lipschitz map $f:A\to H$ defined on a subset $A$ of a Hilbert space $H$ extends to a Lipschitz map $\bar f:H\to H$ with the same Lipschitz constant $\Lip(\bar f)=\Lip(f)$. Using this theorem of Kirszbraun, we can prove the following characterization.

\begin{theorem}\label{t:isoH} A metric fractal $(K,\F)$ is isometrically equivalent to a fractal in a Hilbert space $H$ if and only if its underlying metric space $K$ admits an isometric embedding in $H$.
\end{theorem}

Concerning the isometric embeddability of metric spaces into Hilbert spaces we have the following characterization.

\begin{proposition}\label{p:isoH} For a metric space $X$ with metric $d_X$ the following conditions are equivalent:
\begin{enumerate}
\item $X$ is isometric to a subset of a Hilbert space;
\item $\sum_{i,j=1}^nd^2_X(x_i,x_j)c_ic_j\le0$ for any points $x_1,\dots,x_n\in X$ and real numbers $c_1,\dots,c_n$ with $\sum_{i=1}^nc_i=0$;
\item $\sum_{i,j=1}^n\big(d_X^2(x_i^+,x_j^+)+d_X^2(x_i^-,x_j^-)-2d_X^2(x_i^+,x_j^-)\big)\le0$ for any points $x_1^+,\dots,x_n^+$ and $x_1^-,\dots,x_n^-$ in $X$.
\end{enumerate}
\end{proposition}
  
\begin{proof} The equivalence $(1)\Leftrightarrow(2)$ is proved in \cite[8.5(ii)]{BL}.
To show that $(2)\Ra(3)$, apply (2) to the points $x_1=x_1^+$, \dots, $x_n=x_n^+$, $x_{n+1}=x_1^-$, \dots, $x_{2n}=x_n^-$ and numbers $c_1=\dots=c_n=1$ and $c_{n+1}=\dots=c_{2n}=-1$.

To prove that $(3)\Ra(2)$, assume that (2) does not hold and find points $x_1,\dots,x_n\in X$ and real numbers $c_1,\dots,c_n$ such that $\sum_{i=1}^nc_i=0$ and $\sum_{i,j=1}^nd_X(x_i,x_j)^2c_ic_j>0$. By the continuity of the arithmetic operations, we can assume that the numbers $c_1,\dots,c_n$ are rational. Multiplying these rational numbers by their common denominator, we can make them integer. Repeating each point $x_i$ \ $|c_i|$ times, we can assume that the numbers $c_1,\dots,c_n$ belong to the set $\{-1,1\}$. In this case the equality $\sum_{i=1}^nc_i=0$ implies that $n$ is even and hence $n=2k$ for some $k$. After a suitable permutation, we can assume that $c_i=1$ for $i\le k$ and $c_i=-1$ for $i>k$. Now put $x_i^+=x_i$ and $x_i^-=x_{i+k}$ for $i\le k$ and conclude that 
$$0<\sum_{i,j=1}^nd^2_X(x_i,x_j)c_ic_j=\sum_{i,j\le k}\big(d^2_X(x_i^+,x_j^+)+ d^2_X(x_i^-,x_j^-)-2d_X^2(x_i^+,x_j^-)\big),$$
which contradicts (3).
\end{proof}

A metric space $X$ is called an {\em ultrametric space} if its metric $d_X$ is an {\em ultrametric}, which means that it satisfies the strong triangle inequality
$$d_X(x,z)\le \max\{d_X(x,y),d(y,z)\}\mbox{ \ for all $x,y,z\in X$}.$$

Applying Proposition~\ref{p:isoH}, we obtain the following generalization of the embeddability result of Vestfrid and Timan \cite{VT}, \cite{VT2}.

\begin{corollary}\label{c:ultraH} Each ultrametric space $X$ is isometric to a subset of a Hilbert space.
\end{corollary}

\begin{proof} Given any points $x_1^+,\dots,x_n^+$ and $x_1^-,\dots,x_n^-$ in $X$, we need to check the inequality (3) of Proposition~\ref{p:isoH}. This will be proved by induction on $n$. For $n=1$ the inequality (3) is trivially true. Assume that for some $n\in\IN$ we have proved that the inequality (3) is true for any points
$x_1^+,\dots,x_n^+$ and $x_1^-,\dots,x_n^-$ in $X$. Choose any points $x_1^+,\dots,x_{n+1}^+$ and $x_1^-,\dots,x_{n+1}^-$. After a suitable permutation, we can assume that $d^2_X(x_{n+1}^+,x_{n+1}^-)=\min_{i,j}d^2_X(x_i^+,x_j^-)$.
By the inductive assumption, $$\sum_{i,j\le n}\big(d_X^2(x_i^+,x_j^+)+d_X^2(x_i^-,x_j^-)-2d_X^2(x_i^+,x_j^-)\big)\le0.$$For every $i\le n$, the strong triangle inequality for the ultrametric $d_X$ implies
$$d_X^2(x_i^+,x_{n+1}^+)\le\max\{d_X^2(x_i^+,x^-_{n+1}), d_X^2(x_{n+1}^+,x^-_{n+1})\}=d_X^2(x_i^+,x_{n+1}^-)$$and
$$d_X^2(x_i^-,x_{n+1}^-)\le\max\{d_X^2(x_i^-,x^+_{n+1}), d_X^2(x_{n+1}^-,x^+_{n+1})\}=d_X^2(x_i^-,x_{n+1}^+).$$
Then
$$
\begin{aligned}
&\sum_{i,j\le n+1}\big(d_X^2(x_i^+,x_j^+)+d_X^2(x_i^-,x_j^-)-2d_X^2(x_i^+,x_j^-)\big)=\\
&=-2d_X^2(x_{n+1}^+,x_{n+1}^-)+\sum_{i,j\le n}\big(d_X^2(x_i^+,x_j^+)+d_X^2(x_i^-,x_j^-)-2d_X^2(x_i^+,x_j^-)\big)+\\
&+\sum_{i=1}^n\big(2d_X^2(x_i^+,x_{n+1}^+)+2d_X^2(x_i^-,x_{n+1}^-)-
2d_X^2(x_i^+,x_{n+1}^-)-2d_X^2(x^+_{n+1},x^-_i)\big)
\le0.
\end{aligned}
$$
\end{proof}

A metric fractal $(X,\F)$ will be called an {\em ultrametric fractal} if its underlying metric space $X$ is an ultrametric space. Theorem~\ref{t:isoH} and Corollary~\ref{c:ultraH} imply the following embeddability result.

\begin{theorem}\label{t:ultraH} Each ultrametric fractal is isometrically equivalent to a fractal in the Hilbert space $\ell_2$.
\end{theorem}

For bi-Lipschitz and equi-H\"older equivalences we have a bit weaker embeddability result.

\begin{theorem}\label{t:equiH} If the underlying metric space of a metric fractal $(K,\F)$ admits a (bi-Lipschitz) equi-H\"older embedding to a Hilbert space $H$, then for some $n\in\IN$, the fractal $(X,\F^{\circ n})$ is (bi-Lipschitz) equi-H\"older equivalent to a fractal in $H$. 
\end{theorem}

This theorem can be derived from Kirszbraun Theorem 1.12 in \cite{BL} and the following simple lemma.

\begin{lemma} For an equi-H\"older homeomorphism $\varphi:X\to Y$ of metric spaces $(X,d_X)$ and $(Y,d_Y)$,  there are positive real constants $C$ and $\alpha$ such that for any Lipschitz map $f:X\to X$ the map $\varphi\circ f\circ\varphi^{-1}:Y\to Y$ has Lipschitz constant $$\Lip(\varphi\circ f\circ\varphi^{-1})\le C\cdot \Lip(f)^\alpha.$$
\end{lemma}

\begin{proof} Since $\varphi$ is equi-H\"older, there are positive constants $c$ and $\alpha$ such that
$$\frac1c\cdot d_X(x,x')^\alpha\le d_Y(\varphi(x),\varphi(x'))\le c\cdot d_X(x,x')^\alpha$$ for all $x,x'\in X$.
This inequality implies that 
$$\frac1{c^{1/\alpha}}d_Y(y,y')^{1/\alpha}\le 
d_X(\varphi^{-1}(y),\varphi^{-1}(y'))\le
c^{1/\alpha}d_Y(y,y')^{1/\alpha}$$for all $y,y'\in Y$.

Take any Lipschitz map $f:X\to X$ and observe that for any points $y,y'\in Y$ we have
\begin{multline*}
d_Y(\varphi\circ f\circ\varphi^{-1}(y),\varphi\circ f\circ\varphi^{-1}(y'))\le
c\cdot d_X(f\circ\varphi^{-1}(y),f\circ\varphi^{-1}(y'))^\alpha\le\\
c\cdot(\Lip(f)\cdot d_X(\varphi^{-1}(y),\varphi^{-1}(y'))^\alpha\le
c\cdot\Lip(f)^\alpha (c^{1/\alpha}d_Y(y,y')^{1/\alpha})^{\alpha}=
c^2\cdot \Lip(f)^\alpha d_Y(y,y')
\end{multline*} 
and hence $\Lip(\varphi\circ f\circ\varphi^{-1})\le C\cdot \Lip(f)^\alpha$ for the constant $C:=c^2$.
\end{proof}

\section{Equi-H\"older embeddings of metric fractals into  Euclidean spaces}

Theorem~\ref{t:isoH} and \ref{t:equiH}  reduce the problem of embedding fractals into Hilbert spaces to the problem of embedding their underlying spaces into Hilbert spaces. In case of Euclidean spaces (i.e., finite-dimensional Hilbert spaces) we have a nice characterization of equi-H\"older embeddability, due to Assouad \cite{A}.

He proved that a separable metric space $X$ admits an equi-H\"older  embedding into a Euclidean space if and only if $X$ has the {\em doubling property}, which means that for some number $N\in\IN$, each subset $S\subset X$ can be covered by $\le N$ sets of diameter $\le\frac12\diam(S)$ (in fact, instead of the constant $\frac{1}{2}$, we can take any positive number $\lambda<1$ and obtain an equivalent definition). It is easy to see that each metric space with the doubling property is separable. The Asouad's characterization implies that the  doubling property is preserved by equi-H\"older equivalences of (separable) metric spaces.

We shall say that a metric fractal has {\em the doubling property} if its underlying metric space has the doubling property.
Combining the Assouad's characterization with Theorem~\ref{t:equiH} we get the following characterization.

\begin{theorem}\label{t:dE} A metric fractal $(K,\F)$ has the doubling property if and only if for some $n\in\IN$ the metric fractal $(K,\F^{\circ n})$ is equi-H\"older equivalent to a fractal in a Euclidean space.
\end{theorem}

For ultrametric fractals we can prove a much better embedding result. A self-map $f$ of a metric space will be called {\em $\e$-contracting} if $\Lip(f)\le\e$. 

\begin{theorem}\label{t:ultraE} Each ultrametric fractal $(X,\F)$ with the doubling property is equi-H\"older equivalent to a fractal in $\IR$. More precisely, for every $\e>0$ the fractal $(X,\F)$ is equi-H\"older equivalent to the attractor of a function system consisting of $|\F|$ many $\e$-contracting self-maps of the real line.
\end{theorem}

\begin{proof} We lose no generality assuming that the ultrametric space $X$ has diameter $\diam(X)\le 1$. Fix any real number $\lambda\in[\frac12,1)$ with $\Lip(\F)\le\lambda$. By the doubling property of $X$, there exists a constant $D$ such that each subset $S\subset X$ can be covered by $\le D$ subsets of diameter $\leq\frac12\, \diam(S)$.

For every $n\in\w$ let $\U_n$ be the family of closed balls of radius $\lambda^n$ in the ultrametric space $X$. The strong triangle inequality implies that for any distinct balls $U,V\in\U_n$ we have $\dist(U,V)>\lambda^n$ where $\dist(U,V)=\inf\{d_X(u,v):u\in U,\;v\in V\}$. Moreover, for any $V\in\U_{n+1}$ there exists a unique $W\in\U_n$ with $V\subset W$.

We claim that for any $n\in\w$ and $W\in\U_n$ the family $\U_{n+1}(W):=\{U\in\U_{n+1}:U\subset W\}$ has cardinality $|\U_{n+1}(W)|\le D$. Indeed, by the choice of $D$, the set $W$ has a cover $\mathcal C$ consisting of $\le D$ sets of diameter $\le\frac12\diam(W)\le\frac12\lambda^n\le\lambda^{n+1}$. Each set $U\in\U_{n+1}(W)$ intersects some set $C_U\in\mathcal\C$ and this set is unique as $\diam(C_U)\le\lambda^{n+1}<\dist(U,V)$ for any $V\in\U_{n+1}\setminus\{U\}$.  This uniqueness implies that $|\U_{n+1}(W)|\le|\C|\le D$.

Given any positive $\e<1$, choose a positive real number $\alpha\le 1$ such that $\lambda^\alpha<\frac{\e}{(1+\e)D}$.
For every $n\in\w$ and $W\in\U_n$ choose inductively a closed interval $I_W\subset\IR$ such that 
\begin{itemize}
\item[(a)] $\diam(I_W)=\lambda^{n\alpha}$;
\item[(b)] $I_U\subset I_W$ for any $U\in\U_{n+1}(W)$; 
\item[(c)] $\dist(I_U,I_V)>\frac{\lambda^\alpha}{\e}\lambda^{n\alpha}$ for any distinct sets $U,V\in\U_{n+1}(W)$.
\end{itemize}
For every $n\in\w$ and $W\in\U_n$ the choice of the family $(I_U:U\in\U_{n+1}(W)\}$ is always possible since $$\lambda^{(n+1)\alpha}(1+\tfrac1\e)|\U_{n+1}(W)|\le \lambda^{n\alpha}\big(\lambda^\alpha(1+\tfrac1\e) D\big)<\lambda^{n\alpha}=\diam (I_W).$$

Now consider the map $\varphi:X\to \IR$ assigning to each point $x\in X$ the unique point of the intersection $\bigcap\{I_U:x\in U\in \bigcup_{n\in\w}\U_n\}$. We claim that $\varphi$ is a bi-Lipschitz embedding of the ultrametric space $(X,d_X^\alpha)$ into the real line. 

Given any distinct points $x,y\in X$, find a unique number $n\in\w$ such that $\lambda^{n+1}<d_X(x,y)\le \lambda^n$. Then $x,y\in W$ for some $W\in\U_{n}$ and hence $$|\varphi(x)-\varphi(y)|\le\diam(I_W)=\lambda^{n\alpha}=\lambda^{-\alpha}\lambda^{(n+1)\alpha}<\lambda^{-\alpha}d_X^\alpha(x,y).$$

Since $d_X(x,y)> \lambda^{n+1}$, the points $x,y$ are contained in distinct sets $U,V\in \U_{n+1}(W)$. Now the condition (c) of the inductive construction ensures that
$$|\varphi(x)-\varphi(y)|\ge\dist(I_U,I_V)\ge \frac{\lambda^\alpha}{\e}\lambda^{n\alpha}\ge  \frac{\lambda^\alpha}{\e}d_X^\alpha(x,y).$$

Therefore, $$
 \frac{\lambda^\alpha}{\e}d_X^\alpha(x,y)\le \frac{\lambda^{(n+1)\alpha}}{\e}\le |\varphi(x)-\varphi(y)|\le \lambda^{n\alpha}<\lambda^{-\alpha}d_X^\alpha(x,y)$$
and the map $\varphi$ is a bi-Lipschitz embeding of the ultrametric space $(X,d_X^\alpha)$ into $\IR$ and an equi-H\"older embedding of the ultrametric space $(X,d_X)$ into $\IR$. 

Observe also that for any contraction $f\in\F$ and any points $x,y\in X$ with $\lambda^{n+1}<d_X(x,y)\le\lambda^n$ for some $n\in\w$, we have $d_X(f(x),f(y))\le\lambda\cdot d_X(x,y)\le \lambda^{n+1}$ and hence 
$$|\varphi\circ f(x)-\varphi\circ f(y)|\le \lambda^{(n+1)\alpha}\le\e\cdot|\varphi(x)-\varphi(y)|,$$
which implies that the map $\varphi\circ f\circ\varphi^{-1}:\varphi(X)\to\varphi(X)\subset\IR$ is $\e$-contracting and hence can be extended to an $\e$-contracting map $\bar f:\IR\to\IR$. Now we see that the family $\bar\F:=\{\bar f:f\in\F\}$  consists of $\e$-contractions of the real line and the ultrametric fractal $(X,\F)$ is equi-H\"older equivalent to the fractal $(\varphi(X),\bar\F{\restriction}\varphi(X))$ in the real line.\end{proof}

In Example~\ref{ex:Kam} we shall construct a simple ultrametric fractal which fails to have the doubling property and hence is not equi-H\"older equivalent to a fractal in a Euclidean space. Nonetheless we do not know the answer to the following open problem.

\begin{problem} Is each metric fractal homeomorphic to a metric fractal with the doubling property?
\end{problem}

\section{The Kameyama pseudometrics on topological fractals}

The problem of metrizability of topological fractals was considered by Kameyama \cite{Kam}. On each topological fractal $(X,\F)$ Kameyama defined a family $(p^\F_\lambda)_{\lambda<1}$ of continuous pseudometrics and proved that the topological fractal $(X,\F)$ is topologically equivalent to a metric fractal if and only if for some $\lambda<1$ the pseudometric $p^\F_\lambda$ is a metric, see Corollary 1.14 in \cite{Kam}.

To give the precise definition of the Kameyama pseudometrics $p^\F_\lambda$, we need to introduce some notation. For a topological fractal $(X,\F)$ let $\id_X:X\to X$ be the identity map of $X$ and $\F^{\circ \w}:=\bigcup_{n\in\w}\F^{\circ n}$ where $\F^{\circ 0}=\{\id_X\}$ and $\F^{\circ n}=\{f_1\circ\cdots\circ f_n:f_1,\dots,f_n\in\F\}$ for $n\in\IN$. For a self-map $f\in\F^{\circ\w}$ let $o(f)=\sup\{n\in\w:f\in\F^{\circ n}\}$.

The Kameyama pseudometric $p^\F_\lambda$ on $X$ is defined by the formula
$$p^\F_\lambda(x,y)=\inf_{f_1,\dots,f_n}\sum_{i=1}^n\lambda^{o(f_i)}$$where the infimum is taken over all sequences $f_1,\dots,f_n\in\F^{\circ\w}$ such that $x\in f_1(X)$, $y\in f_n(X)$ and $f_i(X)\cap f_{i+1}(X)\ne\emptyset$ for all $i<n$. In the definition of $p^\F_\lambda$ we assume that $\lambda^{o(f)}=0$ if $o(f)=\w$.

If $(X,\F)$ is a metric fractal and $\lambda\ge\Lip(\F)$, then $\diam(f(X))\le\lambda^{o(f)}\cdot\diam(X)$ for all $f\in\F^{\circ\w}$ and hence $d_X(x,y)\le p^\F_\lambda(x,y)\cdot\diam(X)$, which implies that the psudometric $p^\F_\lambda$ is a metric. On the other hand, if for some $\lambda<1$ the pseudometric $p^\F_\lambda$ is a metric, then it turns $(X,\F)$ into a metric fractal as each map $f\in\F$ has Lipschitz constant $\le\lambda$ with respect to the Kameyama metric $p^\F_\lambda$ (see Proposition 1.12 in \cite{Kam}).

Applying Theorem~\ref{t:dE} to this metric fractal, we obtain the following embeddability criterion.

\begin{corollary} Let $(X,\F)$ be a topological fractal. If for some $\lambda<1$ the Kameyama pseudometric $p^\F_\lambda$ on $X$ is a metric with the doubling property, then for some $n\in\IN$ the topological fractal $(X,\F^{\circ n})$ is topologically equivalent to a fractal in a Euclidean space.
\end{corollary}

Unfortunately, even for fractals in a Euclidean space the Kameyama metric needs not have the doubling property.

\begin{example}\label{ex:Kam} Fix any complex number $c$ such that $|c|=1$ and $c^n\ne 1$ for any $n\in\IN$. On the complex plane $\IC$ consider the system $\F=\{f_1,f_2,f_3\}$ consisting of three contracting self-maps, defined by the formulas $$f_1(z)=\tfrac12z,\;\;f_2(z)=\tfrac{1}{2}cz,\;\; f_3(z)=1\mbox{ for $z\in\IC$}.$$The attractor $X$ of the function system $\F$ coincides with the set
$$\{f(0):f\in\F^{\circ\w}\}=\{0\}\cup\{x_{n,k}:0\le k\le n<\w\}\quad\mbox{where}\quad x_{n,k}=\tfrac1{2^n}c^{k}.$$
\begin{figure}[h]
	\includegraphics[width=10cm]{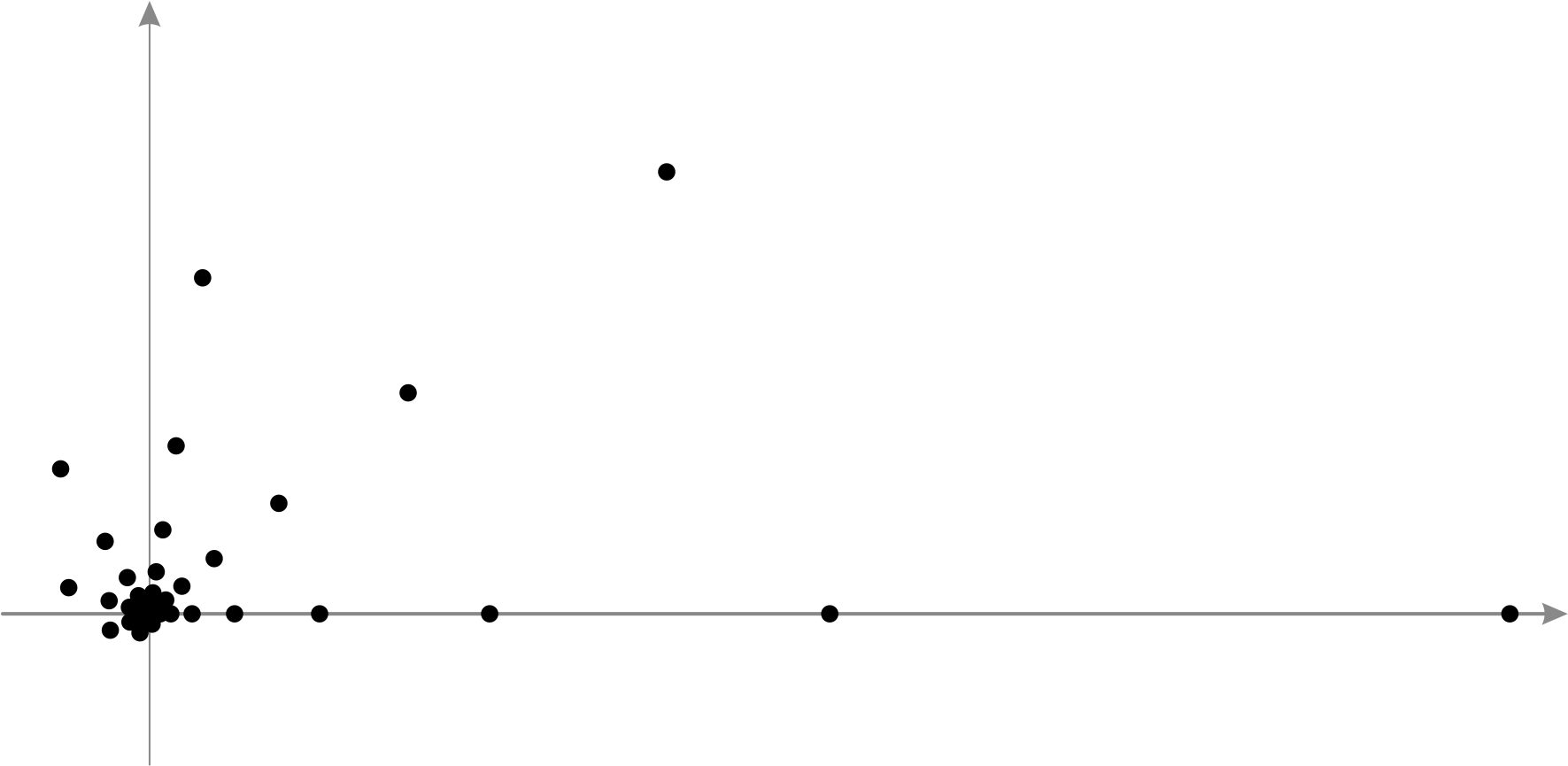}
\end{figure}

Observe that $\{f(X):f\in\F^{\circ\w}\}=\big\{\{x\}:x\in X\setminus\{0\}\big\}\cup\{X_{m,j}:0\le j\le m\}$ where $X_{m,j}=\{0\}\cup\{x_{m+n,k+j}:0\le k\le n\}=f_1^{m-j}\circ f_2^{j}(X)$. Looking at the definition of the Kameyama pseudometric $p^\F_\lambda$, we can see that for any distinct points $x_{n,l},x_{m,q}\in X\setminus\{0\}$ we have
$$p^\F_\lambda(0,x_{n,l})=\lambda^n\mbox{ and }\max\{\lambda^n,\lambda^m\}\le p^\F_\lambda(x_{n,l},x_{m,q})\le \lambda^n+\lambda^m\le 2\max\{\lambda^n,\lambda^m\},$$
which implies that $p^\F_\lambda$ is a metric. Moreover, for various $\lambda<1$ the Kameyama metrics are equi-H\"older equivalent.

To see that the metric $p^\F_\lambda$ does not have the doubling property, observe that for any $n\in\IN$ and distinct numbers $l,q\le n$ we have $\lambda^n<p^\F_\lambda(x_{n,l},x_{n,q})\le 2\lambda^n$, which implies that the set $S_n=\{x_{n,l}:0\le l\le n\}$ cannot be covered by less than $(n+1)$ sets of diameter $\le\frac12\,\diam(S_n)\le\lambda^n$.

Therefore the compact metric space $X_\lambda=(X,p^\F_\lambda)$ does not have the doubling property and the metric fractal $(X_\lambda,\F)$ is not equi-H\"older equivalent to a fractal in a Euclidean space.
On the other hand, the fractal $(X_\lambda,\F)$ is bi-Lipschitz equivalent to a fractal in the Hilbert space $\ell_2$.

This follows from Theorem~\ref{t:ultraH} as the metric $p^\F_\lambda$ is bi-Lipschitz equivalent to the ultrametric $u_\lambda$ on $X$ defined by
$$u_\lambda(0,x_{n,l})=u_\lambda(x_{n,l},0)=\lambda^n\mbox{ \ and \ }u_\lambda(x_{n,l},x_{m,q})=\max\{\lambda^n,\lambda^m\}$$for distinct points $x_{n,l},x_{m,q}\in X\setminus\{0\}$. Let $U_\lambda$ be the set $X$ endowed with the ultrametric $u_\lambda$. It is easy to see that each map $f\in\F$ remains contracting with respect to the ultrametric $u_\lambda$, so $(U_\lambda,\F)$ is a ultrametric fractal. By Theorem~\ref{t:ultraH}, the ultrametric fractal $(U_\lambda,\F)$ is isometrically equivalent to a fractal in $\ell_2$. Then the metric fractal $(X_\lambda,\F)$, being bi-Lipschitz equivalent to the ultrametric fractal $(U_\lambda,\F)$ is bi-Lipschitz equivalent to a fractal in $\ell_2$. On the other hand, the ultrametric space $U_\lambda$ does not have the doubling property (since this property is preserved by bi-Lipschitz homeomorphisms), so $U_\lambda$ is not equi-H\"older equivalent to a subset of a Euclidean space. 
\end{example}

\section{Topological ultrafractals} In this section we describe a class of topological fractals (called topological ultrafractals), for which the Kameyama pseudometrics $p^\F_\lambda$ are ultrametrics with the doubling property.
 
 A topological fractal $(X,\F)$ is called a {\em topological ultrafractal} if for any functions $f,g\in\F^{\circ\w}=\bigcup_{k\in\w}\F^{\circ k}$ one of the following holds:
 $$f(X)\cap g(X)=\emptyset\mbox{ or } f(X)\subset g(X)\mbox{  or }g(X)\subset f(X).$$ In this case the topological fractal structure $\F$ is called a {\em topological ultrafractal structure}.
 
Let us observe that the ultrametric fractal $(U_\lambda,\F)$ constructed in Example~\ref{ex:Kam} is not a topological ultrafractal.

 \begin{theorem}\label{t:Kam}  If $(X,\F)$ is a topological ultrafractal, then for every positive $\lambda<1$ the Kameyama pseudometric $p^\F_\lambda$ is a ultrametric with the doubling property and
 $$p^\F_\lambda(x,y)=\inf\{\lambda^{o(f)}:x,y\in f(X)\mbox{ for some }f\in\F^{\circ \w}\}$$for any $x,y\in X$.
 \end{theorem}
 
 \begin{proof} The definition of $p^\F_\lambda$ implies that $$p^\F_\lambda(x,y)\le u_\lambda(x,y):=\inf\{\lambda^{o(f)}:x,y\in f(X)\mbox{ for some }f\in\F^{\circ\w}\}.$$
 Assuming that $p^\F_\lambda(x,y)<u_\lambda(x,y)$, we can find a sequence of functions $f_1,\dots,f_n\in\F^{\circ\w}$ such that $x\in f_1(X)$, $y\in f_n(X)$, $f_i(X)\cap f_{i+1}(X)\ne\emptyset$ for all $i<n$, and $\sum_{i=1}^n\lambda^{o(f_i)}<u_\lambda(x,y)$. We can assume that the number $n$ is the smallest possible. The definition of $u_\lambda(x,y)$ implies that $n>1$. Since $f_1(X)\cap f_2(X)\ne\emptyset$, either $f_1(X)\subset f_2(X)$ or $f_1(X)\supset f_2(X)$. In both cases we can replace the sequence $f_1,\dots, f_n$ by the shorter sequence $f_2,\dots,f_n$ or $f_1,f_3,\dots,f_n$, which contradicts the minimality of $n$.
 This contradiction shows that $p^\F_\lambda(x,y)=u_\lambda(x,y)$.
 
 Taking into account that the family $\F$ is topologically contracting, we can prove that $u_\lambda$ is a metric. To see that $u_\lambda$ is an ultrametric, take any pairwise distinct points $x,y,z\in X$. Since $u_\lambda(x,y)>0$, there exists a function $f\in\F^{\circ\w}$ such that $u_\lambda(x,y)=\lambda^{o(f)}$ and $x,y\in f(X)$. By the same reason, there exists a function $g\in\F^{\circ\w}$ with 
 $u_\lambda(y,z)=\lambda^{o(g)}$ and $y,z\in g(X)$. Since $(X,\F)$ is a topological ultrafractal and $y\in f(X)\cap g(X)\ne\emptyset$, either $f(X)\subset g(X)$ or $g(X)\subset f(X)$. In the first case we obtain $x,z\in f(X)\cup g(X)\subset g(X)$ and hence $u_\lambda(x,z)\le \lambda^{o(g)}\le \max\{u_\lambda(x,y),u_\lambda(y,z)\}$.  In the second case, $x,y\in f(X)$ and hence $u_\lambda(x,z)\le \lambda^{o(f)}\le \max\{u_\lambda(x,y),u_\lambda(y,z)\}$.
 
 Finally, we show that the ultrametric $u_\lambda$ has the doubling property. It suffices to check that any set $S\subset X$ can be covered by $|\F|$ many sets of diameter $\le\lambda\cdot\diam(S)$. We can assume that $S$ contains more than one point and find $k\in\w$ such that $\diam(S)=\lambda^k$. Fix any point $x\in S$ and observe that $\{f(X):f\in\F^{\circ k},\;x\in f(X)\}$ is a finite cover of $S$. Since $(X,\F)$ is a topological ultrafractal, this cover is linearly ordered and hence contains the largest element $f(X)$. Now we see that $S\subset f(X)=\bigcup_{g\in\F}f\circ g(X)$ and each set $f\circ g(X)$ has diameter $\le\lambda^{k+1}=\lambda\cdot\diam(S)$.
\end{proof}

Theorems~\ref{t:ultraE} and \ref{t:Kam} imply the following embedding result.

\begin{theorem}\label{f:t1} Any topological ultrafractal $(X,\F)$ is topologically equivalent to a fractal in the real line. More precisely, for any positive $\lambda<1$, the  topological ultrafractal $(X,\F)$ endowed with the Kameyama ultrametric $p^\F_\lambda$ is equi-H\"older equivalent to a fractal in the real line. 
\end{theorem}

\section{Strict topological ultrafractals} A topological fractal $(X,\F)$ is called a {\em strict topological ultrafractal} if its fractal structure $\F$ has a property that for any $k\in\w$, any functions $f,g\in\F^{\circ k}$ one of the following holds:
$$\mbox{$f(X)=g(X)$, \; $f(X)\cap g(X)=\emptyset$, \; $|f(X)|=1$ \; or \; $|g(X)|=1$.}
$$
Such a fractal structure $\F$ is called a {\em strict topological ultrafractal structure}.

\begin{proposition}\label{p:strict} Each strict topological ultrafractal $(X,\F)$ is a topological ultrafractal. 
\end{proposition}

\begin{proof} Given distinct functions $f,g\in\F^{\circ \w}$ with $f(X)\cap g(X)\ne\emptyset$, we need to prove that $f(X)\subset g(X)$ or $g(X)\subset f(X)$.

Write $f$ and $g$ as $f=f_1\circ\dots\circ f_n$ and $g=g_1\circ\dots\circ g_m$ for some functions $f_1,\dots,f_n,g_1,\dots,g_m\in \F$. We lose no generality assuming that $n\le m$. Let $\tilde g=g_1\circ\dots\circ g_n$. If $f(X)=\tilde g(X)$, then $g(X)\subset \tilde g(X)=f(X)$ and we are done. So, assume that $f(X)\ne \tilde g(X)$. Taking into account that $(X,\F)$ is a strict topological ultrafractal and $\emptyset\ne f(X)\cap g(X)\subset f(X)\cap\tilde g(X)$, we conclude that $f(X)$ or $\tilde g(X)\supset g(X)$ is a singleton. In the first case $f(X)=f(X)\cap g(X)\subset g(X)$. In the second case $g(X)=g(X)\cap f(X)\subset f(X)$.
\end{proof}

Now we characterize zero-dimensional compact metrizable spaces, homeomorphic to (strict) topological (ultra)fractals. By \cite{BNS} (see also \cite{DS}), a zero-dimensional compact metrizable space $X$ is homeomorphic to a topological fractal if and only if $X$ either is uncountable or is countable and has non-limit scattered height.

Let us recall that a topological space $X$ is {\em scattered} if each subspace $A\subset X$ has an isolated point. The Baire Theorem guarantees that each countable complete metric space is scattered. The complexity of a scattered topological space can be measured by the ordinal $\hbar(X)$ defined as follows.

Let $X$ be a topological space. For a subset $A\subset X$, denote by $A^{(1)}$ the set of non-isolated points of $A$. Put $X^{(0)}=X$ and for every ordinal $\alpha$, define the $\alpha$-th derived set $X^{(\alpha)}$ by the recursive formula $$X^{(\alpha)}=\bigcap_{\beta<\alpha}\big(X^{(\beta)}\big)^{(1)}.$$
The intersection $X^{(\infty)}=\bigcap_{\alpha}X^{(\alpha)}$ of all derived sets has no isolated points and is called the {\em perfect kernel} of $X$. For a scattered topological space $X$, the perfect kernel $X^{(\infty)}$ is empty and the ordinal $\hbar(X)=\min\{\alpha\colon X^{(\alpha)} \mbox{ is finite}\}$ is called the {\em scattered height} of $X$. A scattered topological space $X$ is called {\em unital} if the set $X^{(\hbar(X))}$ is a singleton.

\begin{theorem} For a zero-dimensional compact metrizable space $X$ the following conditions are equivalent:
\begin{enumerate}
\item $(X,\F)$ is a topological fractal for some function system $\F$;
\item $(X,\F)$ is a strict topological ultrafractal for some function system $\F$ of cardinality $|\F|=3$;
\item for every $\e>0$, the space $X$ is homeomorphic to the attractor of a function system consisting of three $\e$-contractions of the real line;
\item $X$ is either uncountable or else $X$ is countable and has non-limit scattered height.
\end{enumerate}
\end{theorem}

\begin{proof} The equivalence $(1)\Leftrightarrow(4)$ was proved in \cite{BNS}, $(2)\Ra(3)$ follows from Theorem~\ref{f:t1} and Proposition~\ref{p:strict}; the implication $(3)\Ra(1)$ is trivial. So, it remains to prove that $(4)\Ra(2)$. In \cite[Lemma 3]{BNS} (more precisely, in its proof) it was shown that a compact metrizable space $X$ has a strict  topological  ultrafractal structure $\F$ consisting of two maps if $X$ either zero-dimensional and uncountable or $X$ is countable and unital with non-limit scattered height. So, it remains to consider the case of a countable non-unital space $X$. In this case $X$ is homeomorphic to the product $Z\times D$ of a unital countable space $Z$ and a discrete space $D$ of cardinality $|D|=|X^{(\hbar(X))}|$. By the preceding case, the countable unital space $Z$ has a strict topological  ultrafractal structure $\F_2$ consisting of two maps. The following lemma implies that the space $X\cong Z\times D$ has a 
strict topological ultrafractal structure $\F$ consisting of three maps. 
\end{proof}

\begin{lemma} If a compact metrizable space $X$ has a strict  topological ultrafractal structure $\F$, then for any non-empty finite discrete space $D$ the space $Y:=X\times D$ has a strict topological ultrafractal structure of cardinality $\le |\F|+1$.
\end{lemma}

\begin{proof} Write $D$ as $D=\{y_1,\dots,y_n\}$ for $n=|D|$. There is nothing to prove for $|D|=1$, so assume that $n> 1$. Replacing $\F$ by a suitable subfamily, we can assume that $X\ne\bigcup_{f\in \F'}f(X)$ for every proper subfamily $\mathcal F'\subset\F$. In this case the ultrafractality of $\F$ ensures that the sets $f(X)$, $f\in\F$, are pairwise disjoint. Write the family $\F$ as $\F=\{f_1,\dots,f_m\}$ where $m=|\F|$. Using the fact that the system $\F$ is topologically contracting, we can show that each map $f_i\in\F$ has a unique fixed point $x_i\in X$ (as was mentioned in Section~\ref{s1}, in such case $f_i$ is a weak contraction with respect to a suitable metric and hence $f_i$ has a unique fixed point by \cite{Edelstein}).

For every positive integer $i\le m$, consider the map $g_i:Y\to Y=X\times D$ defined by the formula
$$g_i(x,y)=\begin{cases}(f_i(x),y_1)&\mbox{if $y=y_1$}\\
(x_i,y_1)&\mbox{otherwise}.
\end{cases}
$$
Also consider the map $g_0:Y\to Y$ defined by $$g_0(x,y)=\begin{cases}
(x,y_{j+1})&\mbox{if $y=y_j$ for some $j<n$};\\
(x_1,y_n)&\mbox{if $y=y_n$},
\end{cases}
$$
It can be shown that $Y=\bigcup_{i=0}^m g_i(Y)$ and that the function system $\mathcal G=\{g_i\}_{i=0}^m$ is topologically contracting.

So, $\mathcal G$ is a fractal structure on $Y$. Let us show that this structure is a strict topological ultrafractal structure. Take any $k\in\IN$ and two functions $\varphi,\psi\in\G^{\circ k}$ such that $\varphi(Y)\ne \psi(Y)$ and $\varphi(Y)\cap \psi(Y)\ne\emptyset$. We should prove that $\varphi(Y)$ or $\psi(Y)$ is a singleton.
 Write $\varphi$ and $\psi$ as $\varphi=g_{p_1}\circ\dots\circ g_{p_k}$ and $\psi=g_{q_1}\circ\dots\circ g_{q_k}$ for some sequences $(p_1,\dots,p_k),(q_1,\dots,q_k)\in\{0,\dots,m\}^k$. If for some $i<k$ we have $p_i\ne0$ and $p_{i+1}=0$, then $g_{p_i}\circ g_{p_{i+1}}(Y)$ is a singleton and so is $\varphi(Y)$. Therefore, we can assume that there exists a non-negative integer numbers $s<k$ such that $p_i=0$ for all $i\le s$ and $p_i\ne 0$ for all $s<i\le k$ (in the case $s=0$, we assume that all $p_i$ are nonzero). If $s\ge n$, then the set $\varphi(Y)\subset g_0^s(Y)=\{(x_1,y_n)\}$ is a singleton and we are done. So, $s<n$.

By analogy we can assume that for some non-negative number $t\le k$ with $t<n$ we have $q_i=0$ for all $i\le t$ and $q_i\ne 0$ for all $t<i\le k$. Since $\varphi(X)\ne \psi(X)$, the sequences $(p_1,\dots,p_k)$ and $(q_1,\dots,q_k)$ are distinct and thus $\min\{s,t\}<k$. Without loss of generality, we can assume that $s\le t$ and hence $s<k$. Then $\varphi(Y)\subset g_0^s(X\times\{y_1\})\subset X\times\{y_{s+1}\}$. Assuming that $s<t$, we conclude that $\psi(Y)\subset g_0^t(Y)\subset X\times\{y_{t+1},\dots,y_n\}$ is disjoint with $\varphi(Y)$. This contradiction shows that $s=t<k$. Since the map $g_0^s|X\times\{y_1\}$ is injective (if $s=0$, then $g_0^0$ is the identity map of $Y$), the sets $g_{p_{s+1}}\circ\dots \circ g_{p_k}(Y)=f_{p_{s+1}}\circ \dots\circ f_{p_k}(X)\times\{y_1\}$ and 
$g_{q_{t+1}}\circ\dots \circ g_{q_k}(Y)=f_{q_{t+1}}\circ \dots\circ f_{q_k}(X)\times\{y_1\}$ are not disjoint and do not coincide. Since $\F$ is a strict  topological ultrafractal structure on $X$, one of these sets is a singleton. Then one of the sets $\varphi(Y)$ or $\psi(Y)$ is a singleton, too.
\end{proof}

\section{Recognizing Euclidean fractals among compact metrizable spaces}

In \cite{DH} Duvall and Husch proved that each finite-dimensional compact metrizable space $X$ containing an open subspace homeomorphic to the Cantor set is homeomorphic to a fractal in a Euclidean space. The following theorem (which is the main technical result of this paper) shows that the Cantor set in the result of Duvall and Husch can be replaced by any uncountable compact zero-dimensional space. 

\begin{theorem}\label{t:main} For a compact metrizable space $X$ containing an open uncountable zero-dimensional subspace $Z$, the following conditions are equivalent:
\begin{enumerate}
\item $X$ is finite-dimensional;
\item $X$ is homeomorphic to a fractal in a Euclidean space;
\item $X$ admits a topological fractal structure $\F$ consisting of 4 contractions such that for some $\lambda<1$ the Kameyama pseudometric $p^\F_\lambda$ is a metric with the doubling property. 
\end{enumerate}
\end{theorem} 

\begin{proof} The implication $(3)\Ra(2)$ follows from Theorem~\ref{t:dE} and $(2)\Ra(1)$ is trivial. It remains to prove that $(1)\Ra(3)$. 

So, assume that the space $X$ is finite-dimensional. By our hypothesis, the space $X$ contains an open uncountable zero-dimensional subspace $Z$. Being uncountable, locally compact and zero-dimensional, the space $Z$ contains an uncountable open compact zero-dimensional subspace. Replacing $Z$ by this subspace, we can additionally assume that the open subspace $Z$ of $X$ is compact and its complement $X\setminus Z$ is not empty. 

Now the idea of further proof is as follows. First we construct a special embedding of the compact zero-dimensional space $Z$ into the square $K\times K$ of the standard Cantor set $K$ such that $K\times\{0\}\subset Z$. The product structure of $K\times K$ will help us to define a topological ultrafractal structure $\F_Z=\{f_0,f_1,f_2\}$ on $Z$ such that for every $\lambda<1$ the Kamyeama ultrametric $p^{\F_Z}_\lambda$ of $(Z,\F_Z)$ induces a standard ultrametric on the Cantor set $K\times\{0\}\subset Z$.

Then we embed the finite-dimensional space $Y:=X\setminus Z$ into a Euclidean space $\IR^d$ and using the partition of $\IR^d$ into cubes,  construct a special surjective map $f_3:Z\to Y$. We shall extend the maps $f_i$, $i\in\{0,1,2,3\}$ to continuous self-maps $\bar f_i$ of $X$ such that $\bar f_i(Y)$ is a singleton and obtain a topological fractal structure $\F=\{\bar f_i\}_{i=0}^3$ on $X$ whose Kameyama pseudometric $p^{\F}_\lambda$ is a metric with doubling property for any $\lambda<1$ with $\lambda^d\ge\frac12$. 

For convenience of the reader, our subsequent proof is divided into 8 steps. 
At the initial two steps we just introduce some notations, necessary for proper handling the Cantor set $K$ and subsets of its square.
\smallskip

\noindent{\bf Step 1: Some notations related to the Cantor cube $2^\w$.} 
 
 By $\w$ we denote the smallest infinite ordinal, which can be identified with the set of non-negative integer numbers. For a non-empty set $A$ by $A^{<\w}:=\bigcup_{n\in\w}A^n$ we denote the set of all finite sequences of elements of the set $A$.
 
For a finite sequence $s=(s_0,\dots,s_{n-1})\in A^n\subset A^{<\w}$ by $|s|$ we denote its length $n$. 
The set $A^0$ is a singleton consisting of the empty sequence, which has length $0$. 

The countable power $A^\w$ consists of infinite sequences {(of length $\w$)} in $A$. 
For an infinite sequence $s=(s_k)_{k\in\w}\in A^\w$ and a number $n\in\IN$ let $s{\restriction}n:=(s_0,\dots,s_{n-1})$ be the restriction of $s$ to the set $n=\{0,\dots,n-1\}$. Similarly we define $s{\restriction}n$ for finite sequence $s$ of length  $|s|\geq n$. 

For two finite sequences $s=(s_0,\dots,s_{n})$ and $t=(t_0,\dots,t_{m})$ by $st:=(s_0,\dots,s_{n},t_0,\dots,t_{m})$ we denote their concatenation.  
 
By $2$ we denote the doubleton $\{0,1\}$. Elements of the set $2^{<\w}\cup 2^\w$ will be called {\em binary sequences}. For $n\in\IN$ by $0^n$ we denote the unique sequence $(0,\dots,0)\in\{0\}^n\subset 2^n$. So, $0^n1$ is the sequence $(0,\dots,0,1)\in 2^{n+1}$ which will be written as $0\dots01$. This is our general rule: writing binary sequences we shall omit commas and parentheses. So, for example, $100001$ will denote the binary sequence $(1,0,0,0,0,1)$.  

For a binary sequence $\alpha\in 2^{<\w}\cup 2^\w$ of length $l\in\w\cup\{\w\}$ let $$\lfloor\alpha\rfloor:=(\alpha_{2n})_{0\le2n<l}\;\;\mbox{and}\;\;\lceil\alpha\rceil:=(\alpha_{2n+1})_{0<2n<l}$$
be the {\em even} and {\em odd parts} of $\alpha$. It is clear that the sequence $\alpha$ can be uniquely recovered from the pair $(\lfloor\alpha\rfloor,\lceil\alpha\rceil)$. If $\alpha$ has finite length, then $$\big|\lceil\alpha\rceil\big|\le\big|\lfloor\alpha\rfloor\big|\le\big|\lceil\alpha\rceil\big|+1.$$

For $i\in\IN$ by ${}^i\!\lfloor\cdot\rfloor\!^i$ we shall denote the $i$-th iteration of the operation $\lfloor\cdot\rfloor$ on $2^{<\w}$. {For example,  ${}^1\!\lfloor 110010011\rfloor\!^1=10101$, ${}^2\!\lfloor 110010011\rfloor\!^2=111$, ${}^3\!\lfloor 110010011\rfloor\!^3=11$ and ${}^i\!\lfloor 110010011\rfloor\!^i=1$ for all $i\ge 4$.}
\smallskip

\noindent{\bf Step 2: Some notations related to the Cantor set $K$ and its square $K\times K$.} 

Consider the topological embedding $x_{(\cdot)}:2^\w\to[0,1]$ assigning to each infinite sequence $\alpha\in 2^\w$ the real number $x_\alpha:=\sum_{n=0}^\infty \frac{2\alpha_n}{3^{n+1}}$. The image $K:=\{x_\alpha:\alpha\in 2^\w\}$ is nothing else but the standard Cantor set in $[0,1]$. 

For any finite binary sequence $\alpha\in 2^{<\w}$ let $$K_\alpha:=\big\{x_\beta:\beta\in 2^\w,\;\beta{\restriction}|\alpha|=\alpha\}$$be a basic closed-and-open subset of $K$. Let $x_\alpha:=x_{\alpha0^\w}$ denote the smallest point of the compact set $K_\alpha$. 

Given a finite binary sequence $\gamma\in 2^{<\w}$ of length $|\gamma|$, consider the rectangle
$$\underline{K}^\gamma:=K_\gamma\times K_{0^{|\gamma|}}$$and its upper subrectangle
$$
K^\gamma:=K_{\gamma}\times K_{0^{|\gamma|} 1}.
$$
In particular, for $\gamma=\emptyset\in 2^0$, we get $$\underline{K}^\emptyset=K_\emptyset\times K=K\times K=K\times (K_0\cup K_1)\mbox{ \ and \ }K^\emptyset:=K_\emptyset\times K_{1}=K\times K_{1}.$$

It is easy to see that $K\times K=(K\times\{0\})\cup\bigcup_{\gamma\in 2^{<\w}}K^\gamma$.

\begin{figure}[h]
	\includegraphics[width=0.8\textwidth]{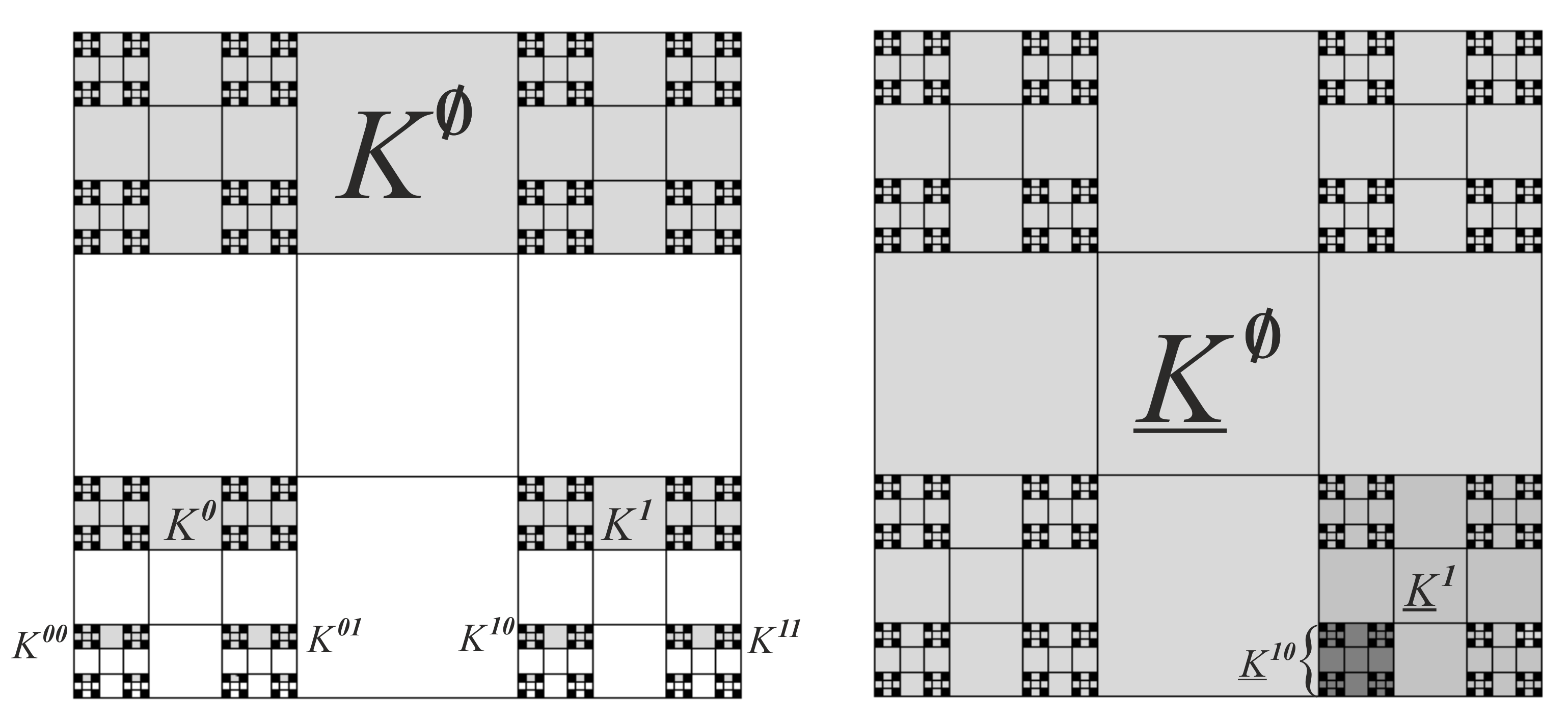}
\end{figure}

\smallskip

For every finite binary sequences $\alpha,\beta\in2^{<\w}$ of length $|\beta|\le|\alpha|\le|\beta|+1$, consider the subrectangle $$
K^\gamma_{\alpha,\beta}:=K_{\gamma \alpha}\times K_{0^{|\gamma|} 1 \beta}.
$$
of the rectangle $K^\gamma$.

\begin{figure}[h]
	\includegraphics[width=0.8\textwidth]{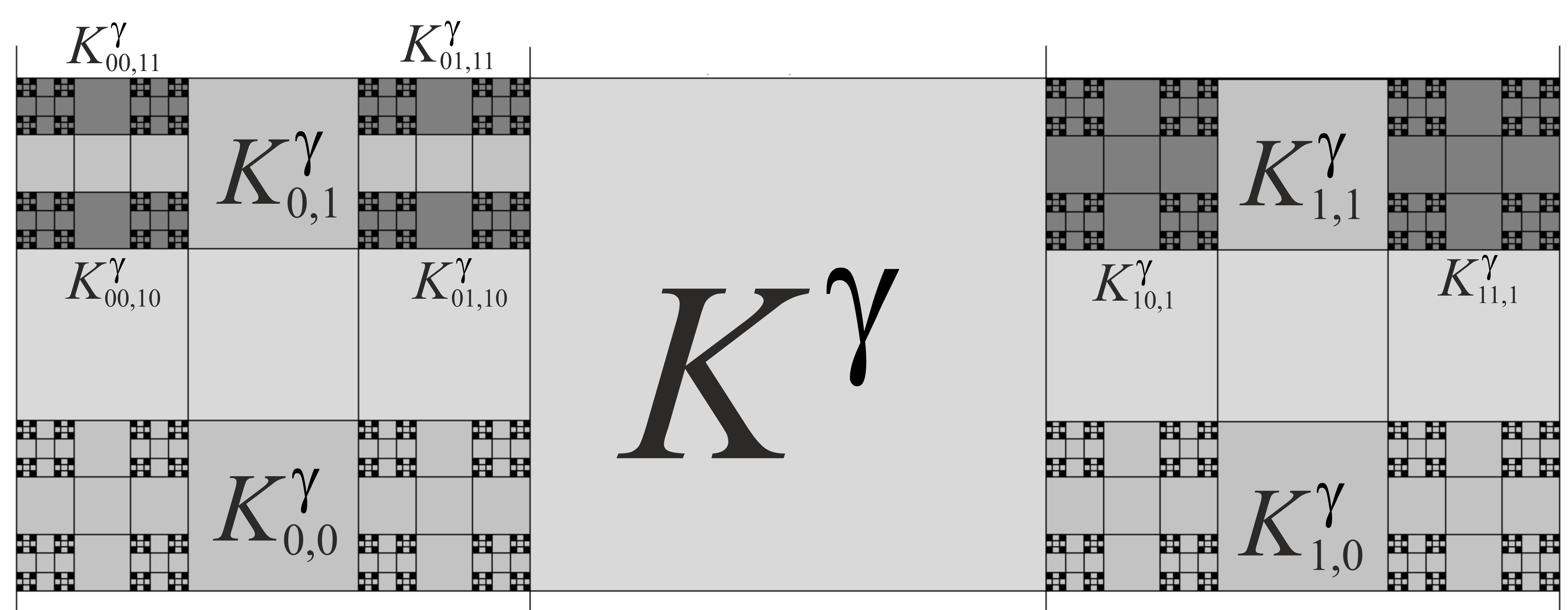}
\end{figure}
\smallskip

\noindent{\bf Step 3: Constructing an appropriate embedding of $Z$ into $K\times K$.} 

Consider the family $\mathcal Z$ of all open countable subsets in $Z$. The union $\bigcup\mathcal Z$, being a metrizable separable space, is Lindel\"of, and hence is countable (being a countable union of open countable subspaces of $Z$).

Then the complement $Z^{(\infty)}:=Z\setminus\bigcup\mathcal Z$ is uncountable and contains no isolated points. Being compact and zero-dimensional, the space $Z^{(\infty)}$ is homeomorphic to the Cantor set $K$ according to the classical Brouwer Theorem \cite[7.4]{Ke}.

In the Cantor set $K$ consider the compact subset
$$S_0:=\{0\}\cup\{\tfrac2{3^n}:n\in\IN\}=\{0\}\cup\{x_{0^n1}:n\in\w\}.$$

\begin{lemma} There exists a topological embedding $h:Z\to K\times K$ such that $h(Z^{(\infty)})= K\times S_0$.
\end{lemma}

\begin{proof} Observe that the product $ K\times S_0$ is homeomorphic to the Cantor set, being a zero-dimensional compact metrizable space without isolated points. Consequently, there exists a homeomorphism $h_0:Z^{(\infty)}\to K\times S_0$. Consider the space $C(Z, K^2)$ of continuous functions from $Z$ to $ K\times K$, endowed with the compact-open topology {(which coincides with the topology of uniform convergence).} 
It is clear that $C_0(Z, K^2):=\{f\in C(Z, K^2):f|Z^{(\infty)}=h_0\}$ is a closed subspace of $C(Z, K^2)$. The space $C_0(Z, K^2)$ is not empty as it contains the composition $h_0\circ r$ of $h_0$ with a retraction $r:Z\to Z^{(\infty)}$ (which exists by \cite[7.3]{Ke}). It is well-known \cite[4.19]{Ke} that the function space 
 $C(Z, K^2)$ is Polish and so is its closed subspace $C_0(Z, K^2)$.

For any point $x\in Z\setminus Z^{(\infty)}=\cup\Z$, consider the open set $U_{x,x}:=\{f\in C_0(Z, K^2):f(x)\notin  K\times S_0\}$ in $C_0(Z, K^2)$. Using the zero-dimensionality of $Z$ and nowhere density of $ K\times S_0$ in $ K\times K$, it can be shown that the open set $U_{x,x}$ is dense in $C_0(Z, K^2)$.

For any distinct points $x,y\in \cup\Z$, consider the open set  $U_{x,y}:=\{f\in C_0(Z, K^2):f(x)\ne f(y)\}$ in $C_0(Z, K^2)$. Using the zero-dimensionality of $Z$ and the fact that $ K^2$ has no isolated points, it can be shown that the set $U_{x,y}$ is dense in $C_0(Z, K^2)$. Since the space $C_0(Z, K^2)$ is Polish, the countable intersection 
$\bigcap_{x,y\in\cup\Z}U_{x,y}$ of the open dense sets in $C_0(Z, K^2)$ is not empty and hence contains some continuous function $h:Z\to  K\times K$.
The inclusion $h\in \bigcap_{x,y\in \cup\Z}U_{x,y}$ implies that the function $h$ is injective and hence is a topological embedding (by the compactness of $Z$).
\end{proof}

From now on, we shall identify the space $Z$ with its image $h(Z)$ in $ K\times K$. Under this identification, the subset $Z^{(\infty)}$ of $Z$ coincides with $ K\times S_0$.

For finite binary sequences $\gamma,\alpha,\beta\in 2^{<\w}$ with $|\beta|\le|\alpha|\le|\beta|+1$ let $$\underline{Z}^\gamma:=Z\cap\underline{K}^\gamma,\quad Z^\gamma:=Z\cap K^\gamma,\quad Z^\gamma_{\alpha,\beta}:=Z\cap K^\gamma_{\alpha,\beta}.$$
The sets $\underline{K}^\gamma:=K_\gamma\times K_{0^{|\gamma|}}$, $K^\gamma:=K_\gamma\times K_{0^{|\gamma|}1}$ and $K^\gamma_{\alpha,\beta}:=K_{\gamma\alpha}\times K_{0^{|\gamma|}1\beta}$ were defined in the second step.

The sets of the form $Z^\gamma_{\alpha,0^{|\alpha|}}$, $Z^\gamma_{\alpha,0^{|\alpha|-1}}$ and $\underline{Z}^\gamma$ are nonempty since they contain parts of $Z^{(\infty)}$. More precisely,
$$
Z^\gamma_{\alpha,0^{|\alpha|-1}}\cap Z^{(\infty)}=Z^\gamma_{\alpha,0^{|\alpha|}}\cap Z^{(\infty)}=K_{\gamma\alpha}\times \big\{\tfrac{2}{3^{|\gamma|+1}}\big\}
\mbox{ \ and \ }
\underline{Z}^\gamma\cap Z^{(\infty)}=K_{\gamma}\times\big(\{0\}\cup\big\{\tfrac{2}{3^{n}}:n>|\gamma|\big\}\big).
$$
On the other hand, the other ``rectangles'' $Z^\gamma_{\alpha,\beta}$ can be empty. 

For any sequences $\gamma,\alpha,\beta\in 2^{<\w}$ with $|\beta|\le|\alpha|\le|\beta|+1$ and non-empty set $Z^\gamma_{\alpha,\beta}$ choose a point $z^\gamma_{\alpha,\beta}\in Z^\gamma_{\alpha,\beta}$ so that $z^\gamma_{\alpha,\beta}=(x_{\gamma\alpha},\tfrac2{3^{|\gamma|+1}})\in Z^{(\infty)}$ if $\beta=0^{|\beta|}$. We recall that by $x_\alpha$ we denote the smallest real number in the set $K_\alpha\subset\IR$. 

\smallskip

\noindent{\bf Step 4: Defining retractions $\ret,\underline{\ret}$ and $\underline{\ret}^\gamma$.} At this step we define three special retractions $\ret:K\times K\to Z$,
$\underline{\ret}:Z\to K\times\{0\}$, and $\underline{\ret}^\gamma:Z^\gamma\to Z^\gamma\cap Z^{(\infty)}$ for all $\gamma\in 2^{<\w}$.

The retraction $\ret:K\times K\to Z$ assigns to each point $z\in K\times K$ the point $\ret(z)$, defined as follows. If $z\in Z$, then put $\ret(z):=z$. If $z\notin Z$, then $z\notin K\times\{0\}$ and there exists a unique $\gamma\in 2^{<\w}$ such that $z\in K^\gamma$. Observe that the family of pairs $$P_z=\big\{(\alpha,\beta)\in 2^{<\w}:|\beta|\le|\alpha|\le|\beta|+1,\;\; z\in K^\gamma_{\alpha,\beta},\;\;Z^\gamma_{\alpha,\beta}\ne\emptyset\big\}$$ is  non-empty (since $(\emptyset,\emptyset)\in P_z$) and 
finite (since $z\notin Z$ and $\diam K^\gamma_{\alpha,\beta}\to 0$ as $\min\{|\alpha|,|\beta|\}\to\infty$). 
 So, we can choose a pair $(\alpha,\beta)\in P_z$ having maximal sum $|\alpha|+|\beta|$ among the pairs in $P_z$ and put $\ret(z):=z^\gamma_{\alpha,\beta}\in Z^\gamma_{\alpha,\beta}$. It is easy to check that the so-defined map $\ret:K^2\to Z$ is a continuous retraction of $K\times K$ onto $Z$.

Next we define the retraction $\underline{\ret}:Z\to K\times\{0\}\subset Z^{(\infty)}$.  For any  point $z\in K\times\{0\}$, put $\underline{\ret}(z):=z$. If $z\in Z\setminus(K\times\{0\})$, then find a unique $\gamma\in 2^{<\w}$ with $z\in Z^\gamma$ and put $\underline{\ret}(z)=(x_\gamma,0)$ where $x_\gamma=\min K_\gamma$.

Finally, for every $\gamma\in 2^{<\w}$ we define a retraction $\underline{\ret}^\gamma:Z^\gamma\to Z^\gamma\cap Z^{(\infty)}=K_\gamma\times\{\frac2{3^{|\gamma|+1}}\}$, which is a ``local'' version of the retraction $\underline{\ret}$. Given any point $z\in Z^\gamma$ define $\underline{\ret}^\gamma(z)\in Z^\gamma\cap Z^{(\infty)}$ as follows. If $z\in Z^{(\infty)}$, then put $\underline{\ret}^\gamma(z):=z$. If $z\notin Z^{(\infty)}$, then we can find unique finite binary sequences $\alpha,\beta\in 2^{<\w}$ such that  $z\in Z^\gamma_{\alpha,\beta}$,  $\beta=0^{|\beta|-1}1$, and $|\alpha|=|\beta|$. In this case we put $\underline{\ret}^\gamma(z):=(x_{\gamma\alpha},\frac{2}{3^{|\gamma|+1}})$.
\begin{figure}[h]
	\includegraphics[width=\textwidth]{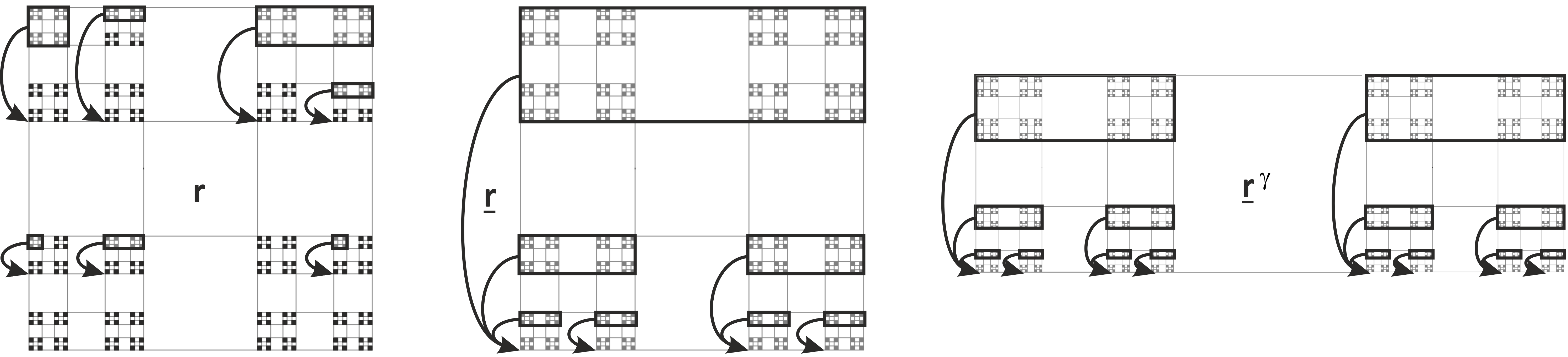}
	\caption{Visualizing the retractions $\ret$, $\underline{\ret}$ and $\underline{\ret}^\gamma$. }
\end{figure}
\smallskip

\noindent{\bf Step 5: Defining the  topological ultrafractal structure $\F_Z=\{f_0,f_1,f_2\}$ on $Z$.} 

First we define the surjective map $f_2:Z\to Z^\emptyset=Z\cap (K\times K_1)$. If $z=(x_\alpha,0)\in K\times\{0\}\subset Z$ for some $\alpha\in 2^\w$, then we put
$$
f_2(z):={\ret}(x_{\lfloor\alpha\rfloor},x_{1\lceil\alpha\rceil})
$$
and if $z\in Z^\gamma$ for some $\gamma\in\{0,1\}^{<\omega}$, then we put $f_2(z):=f_2(\underline{\ret}(z))=f_2(x_\gamma,0)$. We recall that by $\lfloor\alpha\rfloor$ and $\lceil\alpha\rceil$ we denote the even and odd parts of the binary sequence $\alpha=(\alpha_0,\alpha_1,\alpha_2,\dots)$. 

Below we list important properties of  the map $f_2:Z\to Z^\emptyset$. We skip the proof as they follows directly from the construction.
\begin{itemize}
\item[($A_2$)] Let $\gamma\in 2^{<\omega}$ and let $p:=\max\{k\leq|\gamma|:Z^\emptyset_{{\lfloor\gamma{\restriction}k\rfloor},\lceil\gamma{\restriction}k\rceil}\neq\emptyset\}$.
\begin{itemize}
\item[($A_{2=}$)] If $p=|\gamma|$, then $f_2(\underline{Z}^\gamma)=f_2(K_\gamma\times\{0\})=Z^\emptyset_{{\lfloor\gamma\rfloor},\lceil\gamma\rceil}$;
\item[($A_{2<}$)] If $p<|\gamma|$, then $f_2(\underline{Z}^\gamma)=\{z^\emptyset_{\lfloor\gamma{\restriction}p\rfloor,\lceil\gamma{\restriction}p\rceil}\}$.
\end{itemize}
\item[($B_2$)] For any $\gamma\in 2^{<\omega}$ the set $f_2({Z}^\gamma)$ coincides with the singleton $\{f_2(x_\gamma,0)\}=\{\ret(x_{\lfloor\gamma\rfloor,1\lceil\gamma\rceil})\}$.
\item[($C_2$)] the map $f_2$ is continuous.
\end{itemize}
\smallskip

Next, for every $i\in\{0,1\}$ we define the surjective map $f_i:Z\to \underline{Z}^i$. Observe that $$Z=(K\times\{0\})\cup\bigcup_{\gamma\in 2^{<\w}}Z^\gamma\mbox{ \ and \ }\underline{Z}^i=(K_i\times\{0\})\cup\bigcup_{\gamma\in 2^{<\w}}Z^{i\gamma}.$$ Fix any point $z\in Z$. If $z\in K\times\{0\}$, then $z=(x_\alpha,0)$ for a unique $\alpha\in 2^\w$ and we define $f_i(z):=(x_{i\alpha},0)$. If $x\in Z^\gamma$ for some $\gamma\in 2^{<\w}$, then $z=(x_{\gamma\alpha},x_{0^{|\gamma|}1\beta})$ for some $\alpha,\beta\in 2^\w$. If $\beta=0^\omega$, then $z=(x_{\gamma\alpha},\tfrac2{3^{|\gamma|+1}})\in Z^{(\infty)}$ and we put $$
f_i(z):=\ret\big(x_{i\gamma\lfloor\alpha\rfloor},x_{0^{|\gamma|}01\lceil\alpha\rceil}\big)
.$$Otherwise, we put $f_i(z):=f_i(\underline{\ret}^\gamma(z))$.
\smallskip

Now we list properties of $f_i$ which follow from its definition:
\begin{itemize}
\item[($A_i$)] Let $\gamma,\alpha,\beta\in 2^{<\omega}$ be finite binary sequences of length $|\beta|\le|\alpha|\le|\beta|+1$ such that $Z^\gamma_{\alpha,\beta}\neq\emptyset$, and let $p=\max\{k\leq|\alpha|:Z^{i\gamma}_{\lfloor\alpha{\restriction}k\rfloor,\lceil\alpha{\restriction}k\rceil}\neq\emptyset\}$ and $q:=\max\{k\leq |\beta|:\beta{\restriction}k=0^k\}$.
\begin{itemize}
\item[($A_{i{=}}$)] If $p=|\alpha|$ and $q=|\beta|$, then
$
f_i(Z^\gamma_{\alpha,\beta})=Z^{i\gamma}_{\lfloor\alpha\rfloor,\lceil\alpha\rceil}
$.
\item[($A_{i{\le}}$)] If $p\leq q$ and $p<|\alpha|$, then
$
f_i(Z^\gamma_{\alpha,\beta})$ is the singleton $\{z^{i\gamma}_{\lfloor\alpha{\restriction}p\rfloor,\lceil\alpha{\restriction}p\rceil}\}$.

\item[($A_{i>}$)] If $p>q$  and $q<|\beta|$, then $f_i(Z^\gamma_{\alpha,\beta})$ coincides with the singleton $\{f_i(x_{\gamma\alpha'},\frac2{3^{|\gamma|+1}})\}$ where $\alpha'=\alpha{\restriction}(q+1)$.
\end{itemize}
\item[($B_i$)]  $f_i(\underline{Z}^\gamma)=\underline{Z}^{i \gamma}$ for any $\gamma\in 2^{<\omega}$.
\item[($C_i$)] The map $f_i$ is continuous.
\end{itemize}

\begin{figure}[h]
	\includegraphics[width=0.6\textwidth]{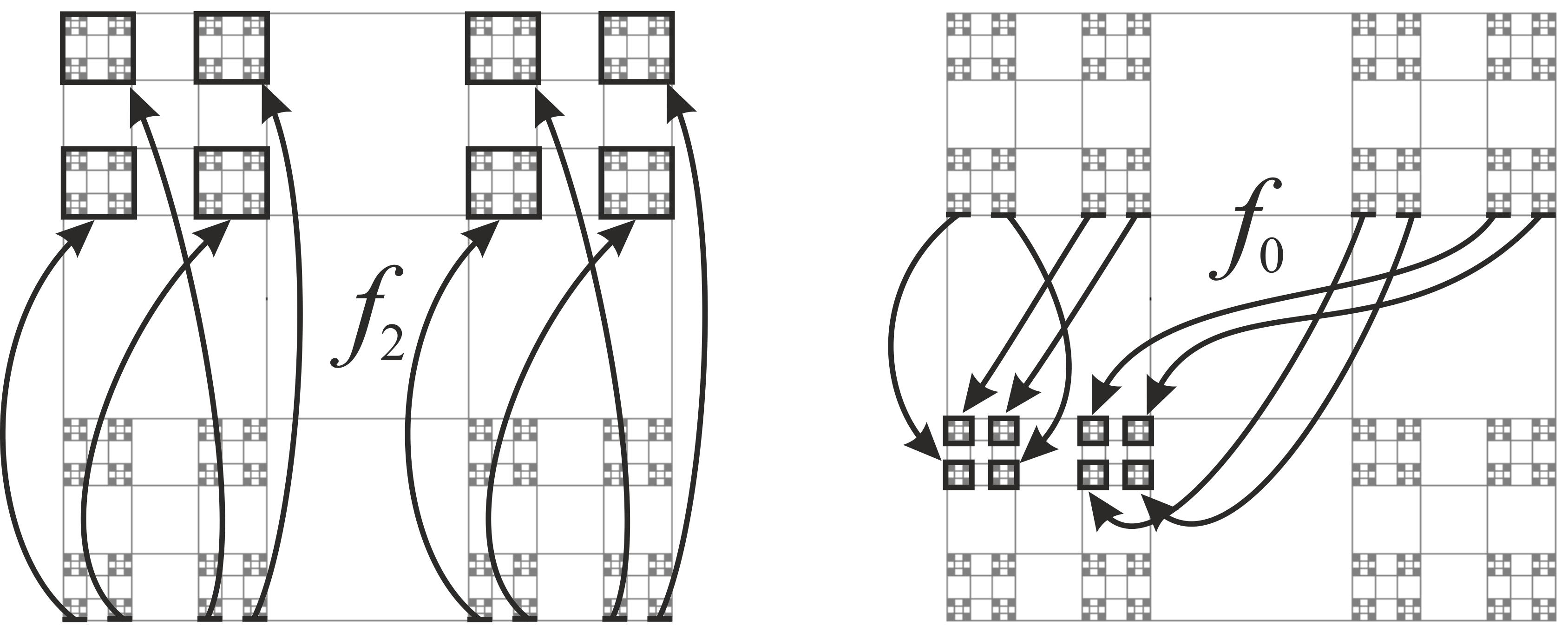}
	\caption{The maps $(x_\alpha,0)\mapsto(x_{\lfloor\alpha\rfloor},x_{1\lceil\alpha\rceil})$ and $(x_{\gamma\alpha},\frac2{3^{|\gamma|+1}})\mapsto(x_{0\gamma\lfloor\alpha\rfloor},x_{0^{|\gamma|}01\lceil\alpha\rceil})$.}
\end{figure}
\smallskip

Now we are going to prove that the space $Z$ endowed with the function family $\F_Z=\{f_0,f_1,f_2\}$ is a strict topological ultrafractal. For this we shall describe the structure of images $f(Z)$ under maps $f\in\F_Z^{\circ\w}$.

Observe that each map $f\in\F_Z^{\circ\w}$ is equal to the composition 
 $$f_\alpha:=f_{\alpha_0}\circ\dots\circ f_{\alpha_n}$$ for a suitable finite sequence $\alpha=(\alpha_0,\dots,\alpha_n)\in \{0,1,2\}^{<\w}$. If $\alpha$ is the empty sequence, then $f_\alpha$ is the identity map of $X$.

Taking into account the properties $(A_i)$--$(C_i)$ of the maps $f_i$, $i\in\{0,1,2\}$, we can prove the following lemma. 

\begin{lemma}\label{l:rec} If for a finite sequence $\gamma\in\{0,1,2\}^{<\w}$ the set $f_\gamma(Z)$ is not a singleton, then one of three possibilities holds:
\begin{enumerate}
\item[\textup{1)}] $\gamma\in\{0,1\}^{<\w}$ and $f_\gamma(Z)=\underline{Z}^\gamma$;
\item[\textup{2)}] $\gamma=2\beta$ for some sequence $\beta\in\{0,1\}^{<\w}$ and $f_\gamma(Z)=Z^\emptyset_{\lfloor\beta\rfloor,\lceil\beta\rceil}$;
\item[\textup{3)}] $\gamma=\alpha2\beta$ for some sequences $\alpha,\beta\in\{0,1\}^{<\w}$ with $|\alpha|>0$ and $f_\gamma(Z)=Z^\alpha_{\lfloor^i\!\lfloor\beta\rfloor\!^i\rfloor,\lceil^i\!\lfloor \beta\rfloor\!^i\rceil}$, where $i=|\alpha|$.
\end{enumerate}
\end{lemma}

Using Lemma~\ref{l:rec} it is not difficult to prove that $(Z,\F_Z)$ is a strict topological ultrafractal. 
\smallskip

By Theorem~\ref{t:Kam}, for every positive $\lambda<1$ the Kameyama pseudometric $p^{\F_Z}_\lambda$ on the topological ultrafractal $(Z,\F_Z)$ is an ultrametric with the doubling property. Moreover, for any points $x,y\in Z$ the distance $p^{\F_Z}_\lambda(x,y)$ can be calculated by the formula 
$$p^{\F_Z}_\lambda(x,y)=\inf\{\lambda^k:x,y\in f(Z)\;\mbox{ for some }f\in\F^{\circ k}\}.$$ Taking into account that $\{f(Z):f\in\F_Z^{\circ\w}\}\subset\big\{\{z\}:z\in Z\big\}\cup\{\underline{Z}^\gamma,Z^\gamma_{\alpha,\beta}:\alpha,\beta,\gamma\in 2^{<\w}\}$, we can conclude that for $x,y\in K\times\{0\}\subset Z$ this formula simplifies to
$$p^{\F_Z}_\lambda(x,y)=\inf\{\lambda^{|\gamma|}:x,y\in \underline{Z}^\gamma\;\mbox{for some $\gamma\in 2^{<\w}$}\}.$$




\smallskip

\noindent{\bf Step 6: Constructing the map $f_3:Z\to X\setminus Z$.} By our assumption, the compact metrizable space $X$ is finite-dimensional and so is the closed-and-open subspace $Y:=X\setminus Z$ of $X$. By the Embedding Theorem 1.11.4 in \cite{End}, the finite-dimensional compact metrizable space $Y$ admits a  topological embedding into the unit cube $[0,1]^d$ in the Euclidean space $\IR^d$ of dimension $d=2\dim(Y)+1$. So, we can (and will) identify $Y$ with a subset of $[0,1]^d$. 

For every $n\in\w$ consider the cover 
$$\square_n:=\big\{\tfrac1{2^n}(x+[0,1]^d):x\in \IZ^d\cap[0,2^n)^d\big\}$$of $[0,1]^d$ by closed cubes with side $\frac1{2^n}$. The family $\square_\w:=\bigcup_{n\in\w}\square_n$ is a tree with respect to the partial order $\le$ defined by $U\le V$ iff $V\subset U$. So, the cube $[0,1]^d$ is the smallest element of the tree $\square_\w$. It is clear that for every $n\in\w$, any cube $U\in\square_n$ has exactly $2^d$ successors in the tree $\square_{\w}$ (those successors are cubes of the cover $\square_{n+1}$, contained in $U$).

In the tree $\square_\w$ consider the subtree $T$ consisting of cubes $U\in\square_\w$ that have non-empty intersection with the subspace $Y$ of $[0,1]^d$. The choice of $Z$ ensures that the space $Y$ is not empty, so for any $n\in\w$ the $n$-th level $T_n=T\cap\square_n$ of the tree $T$ is not empty.

In the binary tree $2^{<\w}$ consider the subtree $(2^d)^{<\w}:=\bigcup_{n\in\w}2^{dn}$. 

By induction on the tree $(2^d)^{<\w}$ we can construct a surjective map $\varphi:(2^d)^{<\w}\to T$ such that for every $n\in\w$ and every $s\in 2^{dn}\subset(2^d)^{<\w}$ the element $\varphi(s)$ belongs to the $n$-th level $T_n$ of $T$ and the set $\{\varphi(\alpha):\alpha\in 2^{d(n+1)},\;\alpha{\restriction}dn=s\}$ coincides with the set $\{t\in T_{n+1}:\varphi(s)\le t\}$ of successors of $\varphi(s)$ in the tree $T$. This set of successor has cardinality $\le 2^d$, so the inductive construction of $\varphi$ is indeed possible.

The surjective morphism of the trees $\varphi:(2^d)^{<\w}\to T$ determines a well-defined continuous surjective map $\partial\varphi:2^\w\to Y$ assigning to each infinite binary sequence $s\in 2^\w$ the unique point of the intersection $\bigcap_{n\in\w}\varphi(s{\restriction}dn)$ of the decreasing sequence of cubes $\big(\varphi(s{\restriction}dn)\big)_{n\in\w}$. It is easy to see that for any finite binary sequence $s\in 2^{dn}\subset (2^d)^{<\w}$ we have $\partial\varphi({\uparrow}s)=Y\cap\varphi(s)$, where ${\uparrow}s=\{\alpha\in 2^\w:\alpha{\restriction}dn=s\}$.

Let $f_3=\partial\varphi\circ x^{-1}_{(\cdot)}\circ\underline{\ret}:Z\to Y$ 
 be the map assigning to each point $z\in Z$ the point $\partial\varphi(\alpha)$ where $\alpha\in 2^\w$ is the unique sequence such that $(x_\alpha,0)=\underline{\ret}(z)$ . It follows that for every $s\in 2^{dn}\subset (2^d)^{<\w}$ the set $\underline{Z}^s$ has image
$$f_3(\underline{Z}^s)=\partial\varphi({\uparrow}s)=Y\cap\varphi(s).$$
\smallskip

\noindent{\bf Step 7: Constructing the topological fractal structure $\F=\{\bar f_i\}_{i=0}^3$ on $X$.}  

As we already know, the function system $\F_Z=\{f_0,f_1,f_2\}$ is a topological (ultra)fractal structure on $Z$, which implies that for every $i\in\{0,1,2\}$ the map $f_i$ has a unique fixed point $z_i\in Z$. 

Let $\bar f_i:X\to Z\subset X$ be the (unique) extension of the map $f_i$ such that $\bar f_i(Y)=\{z_i\}$. In this case $\bar f_i(X)=\{z_i\}\cup f_i(Z)=f_i(Z)$. Let $\bar f_3:X\to Y\subset X$ be an extension of the map $f_3:Z\to Y$ such that $\bar f_3(Y)=\{y_3\}$ for some point $y_3\in Y$.

We claim that the function system $\F=\{\bar f_i\}_{i=0}^3$ on $X$ is a topologically contracting. Given any open cover $\U$ on $X$, use the continuity of the maps $\bar f_i$ for $i\in\{0,1,2,3\}$ and find an open cover $\V$ of $Z$ such that for any $V\in\V$ and $i\in\{0,1,2,3\}$ the set $\bar f_i(V)$ is contained in some set $U\in\U$. Since the function system $\F_Z=\{f_i\}_{i=0}^2$ on $Z$ is topologically contracting, there exists $k\in\IN$ such that for any function $f\in\F_Z^{\circ k}$ the set $f(Z)$ is contained in some set $V\in\V$. We claim that for any function $f\in\F^{\circ(k+1)}$ the set $f(X)$ is contained in some set $U\in\U$. This is trivially true if $f(X)$ is a singleton. So, we assume that $f(X)$ is not a singleton.

Since $f\in\F^{\circ(k+1)}$, there exists a sequence $\alpha=(\alpha_0,\dots,\alpha_k)\in \{0,1,2,3\}^{k+1}$ such that $f$ is equal to the map $$\bar f_\alpha:=\bar f_{\alpha_0}\circ \dots\circ \bar f_{\alpha_k}.$$ If for some positive $i\le k$ the number $\alpha_i=3$, then $\bar f_{\alpha_{i-1}}\circ \bar f_{\alpha_i}(X)=\bar f_{\alpha_{i-1}}(\bar f_3(Y\cup Z))=\bar f_{\alpha_{i-1}}(Y)$ is a singleton and so is the set $f(X)$. Since $f(X)$ is not a singleton, $\alpha_i\in\{0,1,2\}$ for all positive $i\le k$. The choice of $k$ ensures that the set $f_{\alpha_1}\circ\dots \circ f_{\alpha_k}(Z)$ is contained in some set $V\in\V$ and the choice of the cover $\V$ guarantees that $\bar f_{\alpha_0}(V)$ is contained in some $U\in\U$. Then
$$f(X)=\bar f_{\alpha_0}\circ \bar f_{\alpha_1}\circ\dots\circ \bar f_{\alpha_k}(Y\cup Z)=\bar f_{\alpha_0}\circ f_{\alpha_1}\circ\dots\circ f_{\alpha_k}(Z)\subset \bar f_{\alpha_0}(V)\subset U.$$  

Therefore the pair $(X,\F)$ is a topological fractal.

\noindent{\bf Step 8: Exploring properties of the Kameyama pseudometric $p_\lambda^{\F}$ on $X$.} In this step we shall show that for a sufficiently large $\lambda<1$ the Kameyama pseudometric $p^\F_\lambda$ on $X$ is a metric with the doubling property. 

Observe that Lemma~\ref{l:rec} and the definitions of the maps $\bar f_3$ and $\underline{\ret}$ imply the following description of the images $f(X)$ for $f\in\bar \F^{\circ\w}$.

\begin{lemma}\label{l:rec2} If for a finite sequence $\gamma\in\{0,1,2,3\}^{<\w}$ the set $\bar f_\gamma(Z)$ is not a singleton, then one of four possibilities holds:
\begin{enumerate}
\item[\textup{1)}] $\gamma\in\{0,1\}^{<\w}$, and $\bar f_\gamma(X)=f_\gamma(Z)=\underline{Z}^\gamma$;
\item[\textup{2)}] $\gamma=2\beta$ for some sequence $\beta\in\{0,1\}^{<\w}$, and $\bar f_\gamma(X)=f_\gamma(Z)=Z^\emptyset_{\lfloor\beta\rfloor,\lceil\beta\rceil}$;
\item[\textup{3)}] $\gamma=\alpha2\beta$ for some sequences $\alpha,\beta\in\{0,1\}^{<\w}$ with $|\alpha|>0$, and $\bar f_\gamma(X)=f_\gamma(Z)=Z^\alpha_{\lfloor^i\!\lfloor\beta\rfloor\!^i\rfloor,\lceil^i\!\lfloor \beta\rfloor\!^i\rceil}$, where $i=|\alpha|$;
\item[\textup{4)}] $\gamma=3\alpha$ for some sequence $\alpha\in\{0,1\}^{<\w}$, and $\bar f_\gamma(X)=\bar f_3(\underline{Z}^\alpha)$.
\end{enumerate}
\end{lemma}

 Fix any positive real number $\lambda<1$ with $\lambda^d\ge \frac12$ and consider the Kameyama pseudometric $p^\F_\lambda$ generated by the topological fractal structure $\F=\{\bar f_i\}_{i=0}^3$ on $X$. 

Observe that for each map $f\in\bigcup_{k=1}^\infty\F^{\circ k}$ the image $f(X)$ is disjoint either with $Y$ or with $Z$. This implies that $p^\F_\lambda(y,z)=1$ for any $y\in Y$ and $z\in Z$.

The restriction of the Kameyama pseudometric $p^\F_\lambda$ to $Z$ coincides with the Kameyama ultrametric $p^{\F_Z}_\lambda$ generated by the  topological ultrafractal structure $\F_Z$. By Theorem~\ref{t:Kam}, $p^{\F_Z}_\lambda=p^\F_\lambda|Z\times Z$ is an ultrametric with the doubling property.

Now we evaluate the restriction of the Kameyama pseudometric $p^\F_\lambda$ to $Y\times Y$ using the metric generated by the equivalent norm $\|(x_i)_{i\in d}\|:=\max_{i\in d}|x_i|$ on the Euclidean space $\IR^d\supset [0,1]^d\supset Y$.

Let $s\le 1$ be the unique positive real number such that $\lambda^d=\frac1{2^s}$.
The following lemma implies that the Kameyama pseudometric $p^\F_\lambda$ is a metric.

\begin{lemma} For any points $x,y\in Y$ we have $p^\F_\lambda(x,y)\ge\lambda^d\cdot\|x-y\|^s$.
\end{lemma}

\begin{proof} To derive a contradiction, assume that $p^\F_\lambda(x,y)<\lambda^d\cdot\|x-y\|^s$ for some points $x,y\in Y$ and using the definition of $p_\lambda^\F$, find a sequence $g_0,\dots,g_m\in \bar\F^{\circ\w}$ such that $x\in g_0(X)$, $y\in g_m(X)$, $g_i(X)\cap g_{i+1}(X)\ne\emptyset$ for all $0\le i<m$, and $\sum_{i=0}^m\lambda^{o(g_i)}<\lambda^d\cdot\|x-y\|^s$. The last inequality implies that $o(g_i)>d$ for all $i$. We can assume that each set $g_i(X)$ is not a singleton, so the number $o(g_i)$ is finite. 

Let $x_0=x$, $x_{m+1}=y$ and for every 
$i\in\{1,\dots,m\}$ let $x_i$ be any point in the intersection $g_{i-1}(X)\cap g_i(X)$. 
Since $x\in Y$, also $x_1,\dots,x_{m+1}\in Y$ and we can apply Lemma~\ref{l:rec2} and conclude that each map $g_i$ is of the form $\bar f_{3\alpha_i}$ for some sequence $\alpha_i\in \{0,1\}^{o(g_i)-1}$. For every $i\le m$ find a unique number $n_i\in\IN$ such that $dn_i\le o(g_i)-1<dn_i+d$ and let $\beta_i=\alpha_i{\restriction}dn_i\in 2^{dn_i}$. Then $g_i(X)=\bar f_{3\alpha_i}(X)\subset \bar f_{3\beta_i}(X)=\bar f_3(f_{\beta_i}(Z))= f_3(\underline{Z}^{\beta_i})=Y\cap \varphi(\beta_i)$, where $\varphi(\beta_i)$ is a cube in the cover $\square_{n_i}$.

Now we see that
$$\lambda^{o(g_i)}\ge \lambda^{dn_i+d}=\lambda^d\frac1{2^{sn_i}}\ge\lambda^d\cdot \|x_i-x_{i+1}\|^s$$
and hence $$\lambda^d\cdot\|x-y\|^s=\lambda^d\cdot\|x_0-x_{m+1}\|^s\le\lambda^d\sum_{i=0}^m\|x_i-x_{i+1}\|^s\le\sum_{i=0}^m\lambda^{o(g_i)},$$ which contradicts the choice of the sequence $g_0,\dots,g_m$.
\end{proof}

Our final lemma completes the proof of Theorem~\ref{t:main}.

\begin{lemma} The Kameyama metric $p^\F_\lambda$ on $X$ has the doubling property.
\end{lemma}

\begin{proof} The doubling property of $p^\F_\lambda$ will follow as soon as we show that each subset $S\subset Y$ can be covered by $17^d$ sets of diameter $\le\lambda\cdot\diam(S)$. This is trivial if $S$ is a singleton. So, we assume that $S$ contains more than one point. Fix any point $x\in S$ and find a unique $n\in\w$ such that $\frac1{2^{ns}}\le\sup_{y\in S}p^\F_\lambda(x,y)<\frac1{2^{(n-1)s}}$.
Let $\U_x=\{U\in\square_n:\exists y\in U,\;\|x-y\|<\frac8{2^n}\}$. It is easy to see that $|\U_x|\le 17^d$. We claim that $S\subset \bigcup\U_x$. 

Given any $y\in S\setminus\{x\}$, find a sequence $g_0,\dots,g_m\in\F$ such that $x\in g_0(X)$, $y\in g_m(X)$, $g_i(X)\cap g_{i+1}(X)\ne\emptyset$ for all $0\le i<m$, and $\sum_{i=0}^m\lambda^{o(g_i)}<\lambda^{-1}p^\F_\lambda(x,y)$.
We can assume that each set $g_i(X)$ is not a singleton, so the number $o(g_i)$ is finite. Since $S\subset Y=\bar f_3(X)$, we can also assume that $o(g_i)>0$.

Let $x_0=x$, $x_{m+1}=y$ and for every 
$i\in\{1,\dots,m\}$ let $x_i$ be any point in the intersection $g_{i-1}(X)\cap g_i(X)$. 
Since $x\in Y$, we can apply Lemma~\ref{l:rec2} and conclude that each $g_i$ is of the form $\bar f_{3\alpha_i}$ for some sequence $\alpha_i\in \{0,1\}^{o(g_i)-1}$. For every $i\le m$, find a unique number $n_i\in\w$ such that $dn_i\le o(g_i)-1<dn_i+d$ and let $\beta_i=\alpha_i{\restriction}dn_i\in 2^{dn_i}$. Then $g_i(X)=\bar f_{3\alpha_i}(X)\subset \bar f_{3\beta_i}(X)=\bar f_3(f_{\beta_i}(Z))=\bar f_3(\underline{Z}^{\beta_i})=Y\cap \varphi(\beta_i)$, where $\varphi(\beta_i)$ is a cube in the cover $\square_{n_i}$. Consequently,
$$
\begin{aligned}
\|x-y\|^s&=\|x_0-x_{m+1}\|^s\le \sum_{i=0}^m\|x_i-x_{i+1}\|^s\le 
 \sum_{i=0}^m\frac1{2^{n_is}}=\sum_{i=0}^m\lambda^{dn_i}
 =\lambda^{-d}\sum_{i=0}^m\lambda^{dn_i+d}\le\\
 &\le\frac1{\lambda^{d}}\cdot\sum_{i=0}^m\lambda^{o(g_i)}<\frac1{\lambda^{d+1}}p^\F_\lambda(x,y)\le \frac1{\lambda^{d+1}2^{(n-1)s}}=\frac{2^{s(2+\frac1d)}}{2^{ns}}\le \frac{8^s}{2^{ns}}
\end{aligned}
$$
and $\|x-y\|< \frac8{2^n}$. Then $y\in \bigcup\U_x$ and hence $S\subset Y\cap \bigcup\U_x$. For every cube $U\in \U_x\subset\square_n$ the set $Y\cap U$ coincides with the set $\bar f_3(\underline Z^{\gamma})=\bar f_3\circ \bar f_{\gamma}(X)$ for a suitable sequence $\gamma\in 2^{dn}$. Now the definition of the pseudometric $p^\F_\lambda$ implies that $Y\cap U$ has $p^\F_\lambda$-diameter $\le \lambda^{o(\bar f_3\circ \bar f_\gamma)}\le\lambda^{1+dn}=\lambda\cdot\lambda^{dn}=\lambda\cdot\frac1{2^{ns}}\le\lambda\cdot\diam(S)$.
 \end{proof}
\end{proof}

\section{Some Open Problems}

In this section we discuss a possible approach to a solution of the following open problem. 

\begin{problem} Is each metric fractal equi-H\"older equivalent to a fractal in $\ell_2$?
\end{problem} 

We can argue as follows. Let $(X,\F)$ be a topological fractal and $\alpha:\F^\w\to X$ be the surjective continuous map assigning to each sequence $(f_n)_{n\in\w}\in\F^\w$ the unique point of the intersection $$\bigcap_{n\in\w}f_0\circ\dots\circ f_n(X).$$

Given a real number $c<1$ on the countable set $\F^{<\w}=\bigcup_{n\in\w}\{n\}\times\F^n$ consider the measure $\mu_c$ such that $\mu_c(\{(n,f)\})=c^n$ for all $n\in\w$ and $f\in\F^n$.

Consider the map $\chi:\F^\w\to L_2(\mu_c)$ assigning to each sequence $\vec f\in \F^\w$ the characteristic function of the set $\{(n,\vec f{\restriction}n):n\in\w\}$.

Let $Z$ be the closed linear span of the set $\{\chi(x)-\chi(y):x,y\in\F^\w,\;\;\alpha(x)=\alpha(y)\}$ in the Hilbert space $L_2(\mu_c)$. Let $H=L_2(\mu_c)/Z$ be the quotient Hilbert space and $q:L_2(\mu_c)\to H$ be the quotient operator. The definition of $Z$ implies the existence of a continuous map $i_c:X\to H$ such that $i_c\circ \alpha=q\circ \chi$. The map $i_c$ induces a continuous pseudometric $p_c$ on $X$ defined by $$p_c(x,y)=\|i_c(x)-i_c(y)\|\mbox{ \ \ for \ \ $x,y\in X$}.$$

\begin{problem} Under which conditions on $(X,\F)$ and $c<1$ the pseudometric $p_c$ is a metric? Is $p_c$ a metric if $X$ admits a metric in which every map $f\in\F$ is contracting with Lipschitz constant $\Lip(f)\le c$?
\end{problem} 

\section{Acknowledgements} The first author would like to express his sincere thanks to W\l odzimierz Holszty\'nski, Bill Johnson, and Tomasz Kania for very helpful comments during a {\tt mathoverflow}-discussion devoted to absolute Lipschitz retracts and their relation to Banach space theory.

\end{document}